\documentclass{amsart}
\usepackage{amssymb}
\usepackage{amsmath,amsfonts,amsthm,amstext}
\usepackage[all]{xy}
\newtheorem{theorem}{Theorem}[section]
\newtheorem{thm}{Theorem}[section]
\newtheorem{lemma}[theorem]{Lemma}
\newtheorem{prop}[theorem]{Proposition}

\newtheorem{corollary}[theorem]{Corollary}
\newtheorem{definition}[theorem]{Definition}
\newtheorem{examples}[theorem]{Examples}
\newtheorem{remark}[theorem]{Remark}

\newtheorem{thmA}{Theorem}

\newtheorem{corA}[thmA]{Corollary}

\numberwithin{equation}{section}
\newcommand{\bet}{\beta^{(2)}}

\newcommand{\N}{\mathbb{N}}
\newcommand{\BZ}{\mathbb{Z}}
\newcommand{\Z}{\mathbb{Z}}

\newcommand{\Q}{\mathbb{Q}}

\newcommand{\C}{\mathcal{C}}
\newcommand{\cE}{\mathcal{E}}
\newcommand{\F}{\mathbb{F}}
\newcommand{\G}{\Gamma}
\newcommand{\Gr}{\mathcal{G}}
\newcommand{\norm}{\triangleleft}
\newcommand{\sn}{\triangleleft}
\newcommand{\limn}{\lim_{n\to\infty}}       
\newcommand{\limii}{\lim_{i\to\infty}}    
\newcommand{\coker}{{ \rm{coker} }}

\def\Y{\Omega}
\def\X{\Psi}
\def\F{{\rm{F}}}
\def\RG{{\rm{RG}}}
\def\DG{{\rm{DG}}}
\def\FP{{\rm{FP}}}
\def\p{\pi}  
\def\I{\mathcal{I}}

\newcommand{\qu}{\backslash\!\backslash}
\newcommand{\chara}{{\rm{char}}}
\newcommand{\deff}{{\rm{def}}}
\newcommand{\vol}{{\rm{vol}}}

\def\Oplus{\oplus}

\begin{document} 

\title[Volume gradients and homology in towers of residually-free groups]{Volume gradients and homology in towers of residually-free groups}

\author{Martin R. Bridson}

\address{Mathematical Institute,
University of Oxford,
Oxford OX1 3LB,
ENGLAND } 
\email{bridson@maths.ox.ax.uk}

\author{Dessislava H. Kochloukova}

\address
{Department of Mathematics, State University of Campinas (UNICAMP), 
Cx. P. 6065,
13083-970 Campinas, SP, Brazil} 
\email{desi@ime.unicamp.br}


\date{1 September 2015}

\keywords{volume gradient, $\ell^2$ betti numbers, residually-free groups}

\begin{abstract} We study the asymptotic growth of homology groups and the cellular
volume of classifying spaces as one
passes to normal subgroups $G_n<G$ of increasing finite index in a fixed finitely generated
group $G$, assuming $\bigcap_n G_n =1$. We focus in particular on finitely presented residually free groups,
calculating their $\ell_2$ betti numbers, rank gradient and asymptotic deficiency.

If $G$ is a limit group and $K$ is any field, then for all $j\ge 1$ the
limit of $\dim H_j(G_n,K)/[G,G_n]$ as $n\to\infty$ exists and is zero except for $j=1$, where it equals $-\chi(G)$.  
We prove a homotopical version of this theorem in which the dimension of $\dim H_j(G_n,K)$
is replaced by the minimal number of $j$-cells in a $K(G_n,1)$; this includes
a calculation of the rank gradient and the
asymptotic deficiency of $G$. Both the homological and homotopical
versions are special cases of general results about the fundamental
groups of graphs of {\em{slow}} groups.

We prove that if a residually free group $G$ is of type $\rm{FP}_m$ but not of type $\rm{FP}_{\infty}$,
then there exists an exhausting filtration by normal subgroups of finite index $G_n$ so that
$\lim_n \dim H_j  (G_n, K) / [G : G_n] = 0 \hbox{ for } j \leq m$. If $G$ is of type $\rm{FP}_{\infty}$,
then the limit exists in all dimensions and we calculate it.  
\end{abstract}

\maketitle

\section{Introduction} 
In this article we study the growth of homology groups and the cellular
volume of classifying spaces as one
passes to subgroups $G_n$ of increasing finite index in a fixed finitely generated
group $G$;
we are particularly interested in finitely presented residually
free groups. For the most part we shall restrict our attention
to {\em{exhausting normal chains}} (also called residual chains), i.e.~we shall assume that
the finite-index subgroups $G_n$ are normal in $G$, are nested $G_{n+1}\subset G_n$, and
that $\bigcap_{n\ge 0} G_n = 1$.
It is easy to see that
if the ambient group $G$ is of type $\rm{FP}_m$ over a field
$K$, then $\dim H_i(G_n,K)/[G:G_n]$ is bounded by a constant; but
does this  ratio always tend to a limit as $[G:G_n]\to\infty$,
and if so, will the limit  be a (significant) invariant of $G$
or merely an artefact of the exhausting normal chain $(G_n)$ that we chose?

The  {\em{approximation theorem}} of W. L\"uck \cite{Lu1} provides
an emphatic answer for fields of characteristic zero: if $G$ is finitely
presented and of type $\rm{FP}_m$ over $\Z$, then $\lim_n 
\dim H_i(G_n,K)/[G:G_n]$ exists for all $i<m$, the limit is independent of
$(G_n)$, and is equal to the $\ell_2$ betti number $\beta_i(G)$. 
(We refer the reader to \cite{L:book} for a comprehensive introduction to  $\ell_2$ invariants, which
originate in the work of Atiyah \cite{atiy}.)
Much less is known in positive
characteristic, although there are some positive results, e.g.~\cite{elek}, \cite{LLS}, \cite{LO}.

In this article we shall consider homology gradients over arbitrary fields. We shall also study the following homotopical
analogues.  Given a group $B$ of type $\F_m$, we define ${\rm{vol}}_m(B)$ to be the least
number of $m$-cells among all classifying spaces $K(B,1)$ with a finite $m$-skeleton. 
For any group $B$ of type $\F_m$ and any field $K$ one has $\dim_K H_m(B, K)\le\vol_m(B)$.
Note that ${\rm{vol}}_k(B)$ is equal to the rank (minimal number of generators) $d(B)$, and
${\rm{vol}}_2(B)$ bounds the
 {\em{deficiency}} $\deff(B)$, which is\footnote{there are different conventions, with many authors taking the opposite sign for the deficiency
 of an individual presentation, and defining the deficiency of the
 group to be the supremum of $|X|-|R|$} the infimum of  $|R | - |X|$ over all finite presentations $\langle X | R \rangle$ of $B_n$. If $(B_n)$ is a chain of finite-index subgroups in $G$, then
the limit $\lim_n d(B_n)/[G:B_n]$ is known as the {\em{rank gradient}} of $G$ with respect to the chain $(B_n)$
and is often denoted $\RG(G, (B_n))$. Lackenby \cite{lack2} initiated the study of rank gradient in connection
with the study of largeness for 3-manifolds \cite{lack1} and a more combinatorial point of view was developed in \cite{ANJ}.
We define the {\em{$m$-dimensional volume gradient}} of $G$ with respect to the chain $(B_n)$  to be $\lim_n \vol_m(B_n)/[G:B_n]$.

We shall calculate these invariants for {\em{residually free groups}}.
Let $\F_r$ denote the free group of rank $r$.
A group $G$ is  {\em{residually free}} if for
every $g\in G\smallsetminus\{1\}$ there exists a homomorphism $\phi_g:G\to\F_2$
such that $\phi_g(g)\neq 1$. The class of finitely generated residually free groups is rather wild, harbouring all manner of pathologies. 
On the other hand, groups that
are {\em{fully residually free}} are closely akin to free groups in many ways.
By definition, a group is {\em{fully residually free}} if for every finite
$S\subset G$ there is a homomorphism $\phi_S:G\to \F_2$ that is injective on
$S$. Such groups are now more commonly called {\em{limit groups}}, following
Sela \cite{sela1}. This remarkable class of groups is the class one constructs  when one attempts to formulate the notion of an ``approximately free group" in various natural ways. For example, they are the finitely generated groups that have the same universal theory as a free group (in the sense of first order logic); they are the groups that arise as Gromov-Hausdorff limits of sequences of marked free groups \cite{CGi}; and they are the groups one obtains by taking limits of {\em{stable}} sequences of homomorphisms from a fixed group to a free group \cite{sela1}. Basic examples of limit groups are the fundamental groups of
closed surfaces of positive genus, free abelian groups, doubles of free groups along maximal cyclic subgroups, and free products of any finite collection
of the foregoing groups. Limit groups have finite classifying spaces.

\begin{thmA}[Volume Gradients for Limit Groups]\label{t:limit}
Let $(B_n)$ be an exhausting chain of finite-index normal
subgroups  in a limit group $G$. Then,
\begin{enumerate}
\item {\em{Rank Gradient:}} $  \frac{d(B_n)}{[G\colon B_n]}\to -\chi(G)$ as $n\to\infty$;
\item {\em{Deficiency Gradient:}} $ \frac{\rm{def}(B_n)}{[G\colon B_n]}\to \chi(G)$ as $n\to\infty$; 
\item $ \frac{\rm{vol}_k(B_n)}{[G\colon B_n]}\to 0$ as $n\to\infty$, for all $k\ge 2$.
\end{enumerate}
\end{thmA}

\begin{corA}[Asymptotic  Homology of Limit Groups] Let $K$ be a field. If $G$ is a limit group and
 $( B_n )$ is an exhausting sequence of normal subgroups of finite index in
  $G$, then
$$
\limn \frac{\dim H_j
(B_n, K)}{[G : B_n]} =
\begin{cases} - \chi(G) &\text{if}\ j=1\\
0&\text{otherwise.}
\end{cases}
$$
\end{corA}

Note that when $\chara(K)=p$ we do not assume that $[G:B_n]$ is a power
of $p$ (cf.~\cite{lack1}).

\begin{corA}\label{c:beta} If $G$ is a limit group, then $\beta^{(2)}_j(G) = 0$ for $j \not= 1$ and $\beta^{(2)}_1(G) = - \chi(G)$.
\end{corA}

The structure theory of limit groups lends itself well to inductive arguments: 
there is a hierarchical structure on the class of such groups; free groups,
free abelian groups and surface groups lie at the bottom level of the hierarchy, and if
one can show that these groups enjoy a certain property then, proceeding by
induction, one can deduce that all limit groups satisfy that property provided that the
property is preserved under the formation of free products and HNN extensions along
cyclic subgroups; see Section \ref{s:limit}. Given the nature of the induction step,
one is drawn naturally into Bass-Serre theory and, in the case of the results stated
above, counting arguments involving double coset decompositions of finite-index subgroups; see sections \ref{s:Hslow} and \ref{s:counting}.

Correspondingly, our proof of Theorem \ref{t:limit}  proceeds via more 
technical results that relate the growth of cellular volume in classifying spaces for
fundamental groups of graphs of groups to the growth in the vertex and edge groups
of the decomposition. The key idea that gives us control in this context is {\em slowness} for groups. A group $G$
of type $\rm{F}$ is {\em{slow above dimension $1$}} if it is residually finite and for every exhausting
normal chain of finite-index normal subgroups
$(B_n)$ there exists a finite $K(B_n,1)$ with  $r_k(B_n)$ $k$-cells such that
$$
\limn  \frac{r_k(B_n)}{[G:B_n]}=0$$
for all $k\ge 2$. We say that
$G$ is {\em{slow}} if it satisfies the additional requirement that the limit exists and is zero for $k=1$ as well. A key result in Section \ref{s:slow} is Proposition \ref{p:good}: 
 If a residually-finite group $G$ is the fundamental group of a finite graph of groups where all of the edge-groups are
slow and all of the vertex-groups are slow above dimension $1$, then $G$ is slow above dimension $1$.

\begin{thmA}\label{t:Aslow} If a residually finite group $G$ of type $\F$ is slow above dimension $1$, then with
respect to every exhausting normal chain 
 $(B_n)$,
\begin{enumerate}
\item Rank gradient:
$$\RG(G, (B_n)) = 
\limn  \frac{d(B_n)}{[G:B_n]}=-\chi(G),$$
\item Deficiency gradient:
$$\DG(G, (B_n)) = 
\limn  \frac{\deff(B_n)}{[G:B_n]}=  \chi(G).$$ 
\end{enumerate}
\end{thmA}

In Section \ref{s:Hslow} we consider a homological condition called $K$-slowness (Definition \ref{k-slow}) and prove a homological analogue of Theorem \ref{t:Aslow}.

\medskip

In the second part of this paper we focus on the class of finitely presented residually
free groups. This class is much wilder than that of limit groups and the structure theory
is correspondingly more awkward.  
Thus the structure of the arguments in the second half of the paper is more subtle and demanding than those in the
first half: there are many layers of arguments using spectral sequences and they draw on finer structural information about the
groups involved. Our starting point is the theorem of Bridson, Howie, Miller and Short \cite{BHMS1} 
which states that a finitely presented group is residually free if and only if it can be
realised as a subgroup of a direct product of finitely many limit groups so that its projection
to each pair of factors is of finite index.   
Our main result concerning residually free groups is the following.

\begin{thmA}\label{t:rf} Let $m\ge 2$ be an integer, let  $G$ be a residually free group of type $\rm{FP}_m$, and let
$\rho$ be the largest integer such that $G$ contains a direct product of $\rho$
non-abelian free groups.
Then, there exists an exhausting sequence   $(B_n)$
so that for all fields $K$,
\begin{enumerate}
\item
if $G$ is not of type  $\rm{FP}_\infty$, then 
$
\lim_{n}\frac{\dim H_i(B_n, K)}{[G\colon B_n]} = 0$ for all $0\le i\le m$;

\item 
if $G$ is of type $\rm{FP}_\infty$ then for all $j\ge 1$,
$$
\lim_{n\to\infty}\frac{\dim H_j(B_n, K)}{[G\colon B_n]} =  
\begin{cases} (-1)^\rho \chi(G)&\text{if $j=\rho$}\\
0&\text{otherwise}.\\
\end{cases}
$$ 
\end{enumerate}
\end{thmA}  

L\"uck's Approximation theorem tells us that when $K$ is a field of
characteristic $0$,  the limits calculated in Theorem \ref{t:rf} are the $\ell_2$ betti numbers of $G$.  
In this case, one knows that the limit is independent of the sequence $(B_n)$, but for fields of positive characteristic
we do not know this, nor do we know if the limit exists for an arbitrary exhausting normal chain in $G$.

We prove Theorem \ref{t:rf} by using the structure theory of residually
free groups to reduce it to a special case of the following result that we hope will have
further applications. The proof of this theorem is presented in Section \ref{s6}; it accounts for almost half 
 the length of this paper.

\begin{thmA}\label{maintheorem} 
Let $G \subseteq G_1 \times \ldots \times G_k$ be a 
subdirect product of residually-finite groups of type $\F$, each of which
contains a normal free subgroup $F_i<G_i$ such that $G_i/F_i$ is 
torsion-free and nilpotent. Assume $F_i\subseteq G\cap G_i$. Let $m < k$ be an integer, let $K$
be a field, and suppose that each $G_i$ is  $K$-slow above dimension
$1$.

If the projection of $G$ to each $m$-tuple of factors 
$G_{j_1} \times \ldots \times G_{j_m} <G$ is of finite index, then
there exists an exhausting normal chain 
$(B_n)$ in $G$ so that for $0 \leq j \leq m$,
$$\lim_{n\to\infty} \frac{ \dim H_j  (B_n, K)}{ [G : B_n]} = 0.$$ 
\end{thmA}
 
Our proof of Theorem \ref{maintheorem} shows that the homology groups $H_j  (B_n, K)$ are finite dimensional 
for $j\le m$. The Weak Virtual Surjections Theorem \cite[Cor.~5.5]{BK} implies that this finiteness holds more generally.
 
One would like to promote Theorem \ref{t:rf} to a theorem about volume gradients, in the spirit of 
Theorem \ref{t:slowLimit}, but for the moment this is obstructed by unresolved conjectures concerning the
relationship between finiteness properties of residually free groups and the projections to $m$-tuples of
factors in their existential envelopes (in the sense of \cite{BHMS2}). However, in low dimensions these
conjectures have been resolved, and that enables us to prove the following theorem, which is the
subject of the final section of this paper.

\begin{thmA}\label{t:rfRG} Every  finitely presented residually free group $G$  that is not a limit group admits an exhausting normal chain  $(B_n)$ with respect to which the rank gradient 
$$\RG(G, ( B_n )) = \lim_{n\to\infty} \frac{ d(B_n)}{ [G : B_n]} = 0.$$ 
Furthermore,
if $G$ is of type $\FP_3$ but is not commensurable with a product of two limit groups,
 $(B_n)$  can be chosen so that the deficiency gradient
$\DG(G, ( B_n )) = 0$.  
\end{thmA}

The recent work of M. Abert and D. Gaboriau \cite{AG} on higher-cost for group actions
recovers parts (1) and (2) of our Theorem A and establishes similar results for larger classes of groups, including mapping class groups. Earlier work of Abert and Nikolov \cite{AN} links
rank gradient to cost in the sense of measurable group theory. 
Part of the original motivation for this work also comes from measurable
group theory. Groups
$G_1$ and $G_2$ are {\em{measure equivalent}} if they admit commuting,
measure-preserving, free actions on the same probability space. Gaboriau \cite{dg1} proved that surface groups
are measure equivalent to free groups  and asked
if this was true of all limit groups. This has been answered in the affirmative for elementarily free groups
by Bridson, Tweedale and Wilton \cite{BTW} but remains open for limit
groups in general. 
Gaboriau  \cite{gab} proved that if $G_1$ is measure equivalent to $G_2$, then
the $\ell_2$ betti numbers of the groups are proportional, i.e. there exists
a constant $c$ such that $\bet_i(G_1) = c\, \bet_i(G_2)$ for all $i\ge 1$.
The $\ell_2$ betti numbers of a finitely generated
free group are zero except  in dimension $1$, where $\bet_1(F_r)  = -\chi(F)$, so Corollary \ref{c:beta} is
compatible with a positive answer to Gaboriau's question.

\noindent{\bf{Acknowledgements.}}
The first author was supported in part by a Senior Fellowship from the EPSRC and by a Wolfson Research Merit Award from the Royal Society.
The second author was supported in part  by FAPESP and CNPq, Brazil. We thank these organisations. We also
thank the organisers of the IHP conference in 2011 where these results were first presented.

\section{Limit groups and Residually Free Groups}\label{s:limit}

In this section we isolate the basic properties of residually free groups and limit groups
that we need in later sections, providing references where the reader unfamiliar with these fascinating 
groups can find more details.

\subsection{$\omega$-residually free towers and splittings of limit groups}

Limit groups have several equivalent definitions, each highlighting a different aspect of their
nature. In the introduction we defined them to be the finitely generated fully residually free groups,
but for the practical purposes of proving our theorems it is most useful to work with one of the
less intuitively-appealing definitions: a {\em{limit group}} is a 
 finitely generated subgroup of an {\em{$\omega$-residually free tower groups}} (abbreviated $\omega$-rft group);
 see
 \cite{sela2} Theorem 1.1
and \cite{KM}.

{\em{$\omega$-rft spaces}} of height $h\in\N$ are defined by an induction on $h$ and, by definition,
an $\omega$-rft  group is
the fundamental group of an $\omega$-rft space.
A height $0$ tower is the 1-point union of a finite collection of circles, closed
hyperbolic surfaces and tori (of arbitrary dimension), except that the closed surface of
Euler characteristic $-1$ is excluded.
An $\omega$-rft space $Y$ of height $h$ is obtained from an $\omega$-rft space $Y_0$ of height $h-1$ by
adding either (1) a torus $T$
 of some dimension, attached to $Y_0$ by identifying a coordinate circle in $T$ with any
loop $c$ in $Y_0$ such that $[c]$ generates a maximal cyclic subgroup of $\pi_1Y_0$,
or (2) a connected, compact surface $S$ that is either a punctured torus or has
Euler characteristic at most $-2$, where the attachment identifies each boundary
component of $S$ with a homotopically non-trivial loop in $Y_0$, chosen 
so that there exists a retraction $r : Y\to Y_0$ sending $\pi_1S$ to a
non-abelian subgroup of $\pi_1Y_0$.

By definition, {\em{the height of a limit group}} $G$ is the minimal height of an 
$\omega$-rft group that has a subgroup
isomorphic to $G$. Limit groups of height $0$ are free products of finitely many free abelian
groups and of surface groups of Euler characteristic at most $-2$. The
Seifert-van Kampen Theorem decomposes an $\omega$-rft group as a 2-vertex
graph of groups with cyclic edge group, where one of the vertices is an $\omega$-rft group
of lesser height, the other is free or free-abelian of finite rank at least $2$, and the
edge groups are cyclic.  Thus  one can apply Bass-Serre theory to decompose an arbitrary limit group 
as follows --- see \cite[Lemma~1.3]{BH}.

\begin{lemma} If $G$ is a limit group of height $h \geq 1$, then $G$ is the fundamental
group of a finite bipartite graph of groups $\Delta$ in which the edge groups are cyclic; the vertex groups
fall into two types corresponding to the bipartite partition of vertices: type (i) vertex groups
are isomorphic to subgroups of a limit group of height $h - 1$; type (ii) vertex groups are
all free or all free-abelian.
\end{lemma}
 
In the first part of this paper, the only properties of a limit group that we shall use are residual
finiteness and the following decomposition property: 

\begin{corollary}\label{allInC}
If a class of groups $\mathcal C$ contains all finitely generated free abelian and surface groups
and is closed under the formation of  amalgamated free products and HNN extensions with cyclic amalgamated
groups, then $\mathcal C$ contains all limit groups.
\end{corollary}

In the second part of the paper we shall need the following additional property of limit groups,
which was established by Kochloukova \cite{Koch}.

\begin{theorem}\label{t:desi1} Every limit group $G$ has a normal subgroup $F$ that is free with $G/F$ torsion-free and nilpotent.
\end{theorem}

\subsection{Residually free groups and subdirect products}

By definition, a group $G$ is residually free if it is
isomorphic to a subgroup of an unrestricted direct product of free groups. 
In general, one requires infinitely many factors in this direct
product, even if $G$ is finitely generated. For example, the 
fundamental group of a closed orientable surface $\Sigma$
is residually free but
it cannot be embedded in a finite direct product of free groups
if $\chi(\Sigma)<0$, since $\pi_1\Sigma$
does not contain $\Z^2$ and is not a subgroup of a free group.
However,  
Baumslag, Myasnikov and Remeslennikov \cite[Corollary 19]{BMR} 
proved that one can force the enveloping product to be finite
at the cost of replacing free groups by 
{\em{limit groups}} (see also \cite[Corollary 2]{KM} and \cite[Claim 7.5]{sela1}).  

Bridson, Howie, Miller and Short \cite{BHMS1}, \cite{BHMS2} characterized  {\em{the finitely presented residually
free groups}} as follows:

\begin{theorem}\label{t:bhms} A finitely presented group $G$ is residually free if and only if
it can be embedded in a direct product of finitely many 
limit groups $G\hookrightarrow D:=\Lambda_1\times\dots\times \Lambda_n$ so that the intersection with 
each factor is non-trivial and the projection $p_{ij}(G)<\Lambda_i\times \Lambda_j$ to each pair of factors is a subgroup of finite index.

Moreover, in these circumstances, there is a subgroup of finite index $D_0<D$ such that $G$ contains
the $(n-1)$-st term of the lower central series of $D_0$.
\end{theorem}

The ``only if" implication in the first sentence of the above theorem was generalized by  Kochloukova \cite{Koch} as follows.

\begin{theorem}\label{t:desi2} Let $G$ be a subdirect product of non-abelian limit groups and let $s\ge 2$ be an integer. If $G$ is of type $\rm{FP}_s$, then
the projection of $G$ to the direct product of each $s$-tuple of these limit groups has finite index. 
\end{theorem}

We shall also need Theorem A of \cite{BHMS1}.

\begin{theorem}\label{t:infty}
Every residually free group of type ${\rm{FP}}_\infty$ is a subgroup of finite index in a direct product of limit groups. 
\end{theorem}

\section{Bass-Serre Theory and Cellular Volume}\label{s:counting}

We assume that the reader is familiar with Bass-Serre theory as laid out in \cite{serre} and with the
more topological interpretation described in \cite{scott-wall}. We recall some of the basic features
of this theory and fix our notation. 

For us, a graph $X$ consists of two sets $V$ (the vertices) and $E$ (the unoriented edges, or 1-cells). There
are maps $\iota:E\to V$ and $\tau:E\to V$, and we allow $\iota(e)=\tau(e)$. We require that the
graph be connected in the sense that the equivalence relation
generated by $[\iota(e)\sim\tau(e) : e\in E]$ has only one equivalence class in $V$.
 There are two sets of groups:
the {\em{vertex groups}} $G_v$,
indexed by $V$, and the {\em{edge groups}} $G_e$, indexed by $E$, together with monomorphisms $\iota_e: G_e\to G_{\iota(e)}$ and $\tau_e:G_e\to
G_{\tau(e)}$. A graph of groups $\mathcal G$ consists of the above data. It is termed finite if $V\cup E$ is finite. 
Serre associates to these data a {\em fundamental group} $G=\pi_1\Gr$ and a left action of $G$ on a tree $\tilde\Gr$ so that (modulo some natural identifications)  the topological
quotient\footnote{for $H<G$, we write $H\qu\tilde{\Gr}$ to denote the quotient graph of groups, which records the isotropy groups and their inclusions as well as
the topological quotient} is $X$ and the pattern of isotropy groups and inclusions corresponds to the original edge and vertex groups $G_v,G_e<G$. 

If $B<G$ is a   
subgroup, then the graph of groups $B\qu \tilde{\Gr}$, which has fundamental
group $B$,   
has vertex
 groups $\{ G^g_v \cap B \mid v\in V,\ BgG_v\in  B
   \backslash G / G_v\}$  and  edge groups  
$\{ G^g_e  \cap B \mid e\in E,\ B g G_e \in B \backslash G /G_e \}.$ 
In particular, if $B$ is normal and of finite index, then for each $v\in V$ there
are $|G/G_vB|$ vertices in $B\qu \tilde{\Gr}$ where the vertex group is a conjugate of $G_v\cap B$,
and for each $e\in E$ there are $|G/G_eB|$ vertices where the edge group is a conjugate
of $G_e\cap B$. 

\subsection{A model for $K(G,1)$ when $G=\pi_1\Gr$}\label{ss:KG1}
Given explicit CW models for the classifying spaces $K(G_v,1)$ and $K(G_e,1)$, one can realise the monomorphisms
$\iota_e$ and $\tau_e$ by cellular maps, $I_e:K(G_e,1)\to K(G_v,1)$ and $T_e:K(G_e,1)\to K(G_v,1)$.
Attaching the ends of $K(G_e,1)\times [0,1]$ to $K(G_{\iota(e)},1)\coprod K(G_{\tau(e)},1)$ by means of these
maps, for each $e\in E$, we obtain an explicit CW model for $K(G,1)$.

Note that for each $k\ge 1$, the set of $k$-cells in $K(G,1)$ is in bijection with the union of the
sets of $k$-cells in $K(G_v,1)\ [v\in V]$ together with the $(k-1)$-cells in $K(G_e,1)\ [e\in E]$, where an open $(k-1)$-cell $c$ in 
$K(G_e,1)$ contributes the open $k$-cell $c\times (0,1)$ to $K(G,1)$.  We single this
simple observation out as a lemma because it is central to what follows. In order to
state this lemma we need the following terminology.

\begin{definition} Let $G$ be a group. A sequence of non-negative integers $(r_k)_{k\ge 1}$
is a {\em{volume vector}} for $G$ if there is a classifying space $K(G,1)$ that,
for all $k\in\N$, has
exactly $r_k$ open $k$-cells.

When we are discussing several groups, we shall abuse notation by writing the
entries of such a vector as $r_k(G)$ but this is {\em{not}} meant to imply that
$r_k(G)$ is an invariant of $G$. 

Note that if $G$ is type $F_\infty$ then $\vol_k(G)$ is the infimum of $r_k(G)$ over all volume vectors for $G$.
\end{definition}

\begin{lemma}\label{lemma1}
 Let $\Gr$ be a finite graph of groups and let $G=\pi_1\Gr$.  With the notation established above, suppose that $(r_k(G_v))$ is a volume vector for $G_v\ (v\in V)$ and 
$(r_k(G_e))$ is a volume vector for $G_e\ (e\in E)$. For $k\ge 1$, let
$$r_k(G):= \sum_{v\in V}r_k(G_v) + \sum_{e\in E} r_{k-1}(G_e).$$
Then $(r_k(G))$ is a volume vector for $G$.
\end{lemma}

We are interested in what happens to $\vol_k(G)$ as we pass to subgroups of increasing index, and we shall do this by constructing suitably-controlled volume vectors.
The most obvious way of getting
 models $K(B_n,1)$ for subgroups $B_n<G$ is to simply take the corresponding covering spaces of a fixed model for $G$. However,
 this model is not efficient enough for our purposes: in general it has too many cells. A simple example that illustrates this is the
 case $G=\Z^r$: the number of $k$-cells in the cover corresponding to $B_n<G$ goes to infinity as $[G:B_n]\to\infty$, but $\vol_k(B_n)$ 
 remains constant since $B_n\cong\Z^r$ for all $n$.
 
 To avoid the phenomenon illustrated by this example, 
 given a finite-index subgroup $B<G$ we first pass to the covering graph-of-groups $B\qu \widetilde\Gr$, where
 $\widetilde\Gr$ is the universal covering (tree) for $\Gr$. 
  In the graph of groups $B\qu \widetilde\Gr$, for each $v\in V$ the vertices lying above
 $v$ are indexed by the double cosets $B\backslash G/G_v$, and the vertex group at the vertex indexed by $BgG_v$ is
 $B^g\cap G_v$. Likewise, the edges of $B\qu \widetilde\Gr$ are indexed by $\coprod_{e\in E} B\backslash G/G_e$
 and the edge groups have the form $B^g\cap G_e$. We now assemble $K(B,1)$ from classifying spaces for the
 edge and vertex groups, as described in  paragraph (\ref{ss:KG1}).  
 If $B$ is normal, then we take the same classifying space above each of
 the vertices indexed by a fixed vertex or edge of $\Gr$.

 We are interested only in the case where $B$ is normal. In that case,
 the above discussion shows that the vertices of the finite graph of groups 
  $ B\qu \widetilde\Gr$ are index by cosets $G/BG_v$, and the edges by $G/BG_e$,
  and from   Lemma \ref{lemma1} we obtain  the following count:
 
 \begin{prop}\label{p:covering} 
Let $G$ be the fundamental group of a finite graph of groups with
 vertex groups $G_v \ (v\in V)$ and edge groups $G_e\ (e\in E)$.
 Let $B<G$ be a normal subgroup of finite index. Suppose that
 volume vectors
 $(r_k(B\cap G_v))$ and $(r_k(B\cap G_e))$ are given and define
 $$
 r_k(B):= \sum_{v\in V}{[G:BG_v]}\, r_k(B\cap G_v) + \sum_{e\in E} {[G:BG_e]}\, r_{k-1}(B\cap G_e).
 $$
 Then $(r_k(B))$ is a volume vector for $B$.
 \end{prop} 
 
 It is clear from the above formula that we will need to count cosets carefully. The
 following trivial observation is useful in this regard.
 
 \begin{remark}\label{l:count}
 Let $G$ be a group, let $H<G$ be a subgroup and let $B\norm G$ be a normal subgroup
 of finite index. Then
 $
 \frac{[G : BH]}{[G:B]} =   \frac{1}{[H:B\cap H]}.
 $
 \end{remark}
  
\section{Volume gradient, slow groups and hierarchies}\label{s:slow}

\begin{definition} A group $G$ of type $\rm{F}$ is {\em{slow above dimension $1$}} if it is residually finite and for every chain of finite-index normal subgroups
$(B_n)$ with $\bigcap_n B_n =\{1\}$, there exist volume vectors $(r_k(B_n))_k$ with only finitely
many non-zero entries, so that
$$
\limn \frac{r_k(B_n)}{[G:B_n]}=0$$
for all $k\ge 2$.

$G$ is {\em{slow}} if it satisfies the additional requirement that the limit exists and is zero for $k=1$ as well.
\end{definition}
 
The following theorem was stated in the introduction as Theorem \ref{t:Aslow}.

\begin{theorem}\label{t:slow} If a residually finite group $G$ of type $\F$ is slow above dimension $1$, then with
respect to every exhausting normal chain 
 $(B_n)$,
\begin{enumerate}
\item Rank gradient:
$$\RG(G, (B_n)) = 
\limn\frac{d(B_n)}{[G:B_n]}=-\chi(G),$$
\item Deficiency gradient:
$$\DG(G, (B_n)) = 
\limn \frac{\deff(B_n)}{[G:B_n]}=  \chi(G).$$ 
\end{enumerate}
\end{theorem}

\begin{lemma}\label{lem2}  Let $G$ be a residually-finite group of type $\rm{F}$ with an exhausting normal
chain $(B_n)$. Suppose that
$G$ is slow above dimension $1$ and choose volume vectors $(r_k(B_n))_k$ as in the
definition. Then 
$$
\limn  \frac{r_0(B_n)- r_1(B_n)}{[G:B_n]} = \chi(G),
$$
and for every field $K$
$$
\limn  \frac{ H_1(B_n, K)}{[G:B_n]} = -\chi(G). 
$$
\end{lemma}

\begin{proof} As $B_n$ is of type $\rm{F}$, it has an Euler characteristic, which may be calculated from 
a finite $K(B_n,1)$ with $r_k(B_n)$ cells of dimension $k$,
$$\chi(B_n) = r_0(B_n) -r_1(B_n) + r_2(B_n) -r_3(B_n) +\dots + (-1)^{k_0}r_{k_0}(B_n).
$$
Euler characteristic is multiplicative in the sense that $\chi(B_n) = [G:B_n]\,\chi(G)$. 
To obtain the first equality, we divide by $[G:B_n]$ and let $n\to\infty$. 

Towards the second equality, observe that 
since the homology of $B_n$ can be computed from the cellular chain complex of $K(B_n,1)$,
we have $r_k(B_n) \ge \dim H_k(B_n, K)$, and therefore 
$\lim_n \dim H_k(B_n, K)/[G:B_n] =0$ for $k\ge 2$ (by slowness). Thus the second equality
can be obtained by calculating $\chi(B_n)$ as the alternating sum of betti numbers (omitting the
coefficients $K$)
$$[G:B_n]\chi(G)=\chi(B_n) = 1 -\dim H_1(B_n)  +  \dots + (-1)^{k_0}\dim H_{k_0}(B_n),
$$
then dividing through by $[G:B_n]$ and taking the limit.
\end{proof} 

\noindent{\bf{Proof of Theorem \ref{t:slow}.}} 
First we prove (1).
Let $r_k(B_n)$ be as in Lemma \ref{lem2}. Now, $B_n$ is a quotient 
of the fundamental group of the 1-skeleton of $K(B_n,1)$, which is a free group of rank $r_1(B_n) - r_0(B_n) +1$, so
$$
\dim H_1(B_n,\Q)\le d(B_n) \le r_1(B_n) - r_0(B_n) +1.
$$
If we divide by $[G:B_n]$ and take the limit, both sides will converge to $-\chi(G)$, by Lemma \ref{lem2}.
\medskip

Turning to the proof of (2), we remind the reader that  the {\em{deficiency}} $\deff(\G)$ is the infimum of $|R | - |X|$ over all possible finite presentations $\langle X \mid R \rangle$ of $\G$.

From the 2-skeleton of any $K(B_n,1)$ we get a group presentation with $r_1(B_n)-r_0(B_n)+1$ generators and
$r_2(B_n)$ relators, so 
$$
r_2(B_n) - r_1(B_n) + r_0(B_n) -1 \ge \deff(B_n).
$$
And from \cite[Lemma~2]{BT} we have  
$$
\begin{aligned}
\deff(B_n) &\geq d(H_2(B_n, \mathbb{Z})) - {\rm{rk}_\Q\ } H_1(B_n, \mathbb{Z})\\ 
&=d(H_2(B_n, \mathbb{Z}))  - \dim H_1(B_n, \mathbb{Q})\\
& \geq \dim H_2(B_n, \mathbb{Q}) - \dim H_1(B_n, \mathbb{Q}).\\
\end{aligned}
$$ 
Thus
$$
r_2(B_n) - r_1(B_n) + r_0(B_n) -1 \ge \deff(B_n) \ge \dim H_2(B_n, \mathbb{Q}) - \dim H_1(B_n, \mathbb{Q}).
$$ 
We divide by $[G:B_n]$ and let $n$ go to infinity. Since $G$ is slow above dimension $1$,  both
$r_2(B_n)/[G:B_n]$ and  $\dim H_2(B_n)/[G:B_n]$ tend to zero, and by  Lemma \ref{lem2} the limit of what remains on
each side tends to $\chi(G)$.
\qed

\begin{examples} {\em Easy examples of slow groups include 
finitely generated torsion-free nilpotent groups. The trivial group is slow. 
Free groups are slow above dimension $1$, trivially.
Surface groups are slow above dimension $1$ because a finite-index subgroup of a surface group
is again a surface group, so for any finite-index subgroup $B$ we have the volume vector $1,d(B),1,0,\dots$}

{\em
A far greater range of examples is provided by the following construction. }
\end{examples}
\def\Sr{\mathcal{S}}

\subsection{Hierarchies} Given a class $\C_0$ of residually-finite groups, we define $\Sr(\C_0) = \bigcup_n\C_n$ inductively
by decreeing that $\C_{n+1}$ consist of all residually-finite groups that can be expressed as the fundamental
group of a finite graph of groups with edge-groups that are slow and vertex-groups that lie in $\C_n$. The
following
proposition  tells us that if the groups in $\C_0$ are slow above dimension $1$, then the groups in 
$\Sr(\C_0)$ are as well.
Note that this includes the
statement that free products of groups that are slow above dimension $1$ are again slow above dimension $1$.

\begin{prop}\label{p:good}
 If a residually-finite group $G$ is the fundamental group of a finite graph of groups where all of the edge-groups are
slow and all of the vertex-groups are slow above dimension $1$, then $G$ is slow above dimension $1$.
\end{prop}

\begin{proof} Given a sequence of finite index normal subgroups $(B_n)$ exhausting $G$, we build
classifying spaces $K(B_n,1)$ as in Proposition \ref{p:covering} so that they have
volume vectors $(r_k(B_n))$ satisfying
\begin{equation}\label{sums}
 r_k(B_n)= \sum_{v\in V}{[G:B_nG_v]} r_k(B_n\cap G_v) + \sum_{e\in E} {[G:B_nG_e]}r_{k-1}(B_n\cap G_e),
\end{equation}
 where $G_v$ and $G_e$ are the vertex and edge groups in the given decomposition of $G$. By hypothesis,
 the $G_v$ are slow above dimension $1$, so for all $k\ge 2$ we have
 $$
\limn  \frac{r_k(B_n\cap G_v)}{[G_v: B_n\cap G_v]} = 0.
 $$
 But, as we noted in Remark \ref{l:count},  $[G_v: B_n\cap G_v] = [G:B_n]/[G:B_nG_v]$, so dividing
 equality (\ref{sums}) by $[G:B_n]$, the first sum   becomes
 $$
 \sum_{v\in V}\frac{r_k(B_n\cap G_v)}{[G_v: B_n\cap G_v]},
 $$ 
 which converges to $0$ as $n\to\infty$. Likewise, since $G_e$ is assumed to be slow (in all dimensions
 including $1$), we have
 $$
 \frac{1}{[G:B_n]}\sum_{e\in E} {[G:B_nG_e]}r_{k-1}(B_n\cap G_e) = 
 \sum_{e\in E} \frac{r_{k-1}(B_n\cap G_e)}{[G_e:B_n\cap G_e]} \to 0$$
 as $n\to\infty$. Thus $\frac{r_k(B_n)}{[G : B_n]} \to 0$ as $n\to\infty$ for all $k\ge 2$, as required.
 \end{proof} 
 
 We saw in Example 4.4 that all finitely generated free groups, surface groups,
and free-abelian groups of finite rank are slow above dimension 1, and
 Corollary \ref{allInC} tells us that if $\C_0$ contains these groups then $\Sr(\C_0)$ contains all {\em limit groups}. 
Thus Proposition \ref{p:good} implies:

\begin{theorem} \label{t:slowLimit} All limit groups are slow above dimension $1$.
\end{theorem}

 \subsection{Proof of Theorem A} Immediate from Theorems \ref{t:slow} and
 \ref{t:slowLimit}.

 \section{Homological slowness} \label{s:Hslow}
 
 In this section we present homological analogues of the results in the previous
 section. The results for residually free groups that are the main focus of this article
 can be deduced using these homological results rather than the homotopical ones, but this is
 not the point of presenting this variation on our earlier theme. The real justification is that
 these homological results apply to more groups. For example, one does not require the groups to
 have classifying spaces with finite skeleta. Also,
 if an amenable group $G$ of type $\F$
is residually finite and the group algebra $K G$ does not have zero divisors, then one knows that $G$ is $K$-slow in the following sense for any field $K$ \cite[Thm.~0.2 (ii)]{LLS} 
but we do not know that it must be slow. Thus one can use amenable groups among the
building blocks for groups built in a hierarchical manner by repeated application
of Proposition \ref{p:HslowGraphs}.
  
\begin{definition} \label{k-slow} Let $K$ be a field and let $G$ be a residually finite group. 
$G$ is {\em{$K$-slow above dimension $1$}} if for every chain of finite-index normal subgroups
$(B_n)$ with $\bigcap_n B_n =\{1\}$,  
$$
\limn  \frac{\dim_K H_j(B_n, K)}{[G:B_n]}=0$$
for all $j\ge 2$.

$G$ is {\em{$K$-slow}} if it satisfies the additional requirement that the limit exists and is zero for $j=1$ as well.
\end{definition}
 
The argument given in  Lemma \ref{lem2} establishes the following result.

\begin{lemma}\label{lem3}  Let $K$ be a field and let $G$ be a residually finite group of type $\rm{F}$ with an exhausting normal
chain   $(B_n)$. If
$G$ is $K$-slow above dimension $1$ then
$$
\limn \frac{ \dim H_1(B_n, K)}{[G:B_n]} = -\chi(G). 
$$
\end{lemma}

\begin{prop}\label{p:HslowGraphs} Let $K$ be a field.
 If a residually-finite groups $G$ is the fundamental group of a finite graph of groups where all of the edge-groups are
$K$-slow and all of the vertex-groups are $K$-slow above dimension $1$, then $G$ is 
$K$-slow above dimension $1$.
\end{prop}

\begin{proof} Given a sequence of finite index normal subgroups $(B_n)$ exhausting $G$, we build
classifying spaces $K(B_n,1)$ as in Proposition \ref{p:covering}. Our aim is to decompose these spaces 
and use the Mayer-Vietoris sequence of this decomposition to establish the following inequality
for all $j\ge 2$ (expressed in the notation of Section \ref{ss:KG1}, with the coefficients $K$ omitted),
\begin{equation}\label{sumsH}
 \begin{aligned}
 \dim_K H_j(B_n)
 \leq 
&\sum_{v\in V}{[G:B_nG_v]}   \dim_K H_j(B_n\cap G_v) +\\
 &\sum_{e\in E} {[G:B_nG_e]} \dim_K H_{j}(B_n\cap G_e) +
  2\sum_{e\in E} {[G:B_nG_e]} \dim_K H_{j-1}(B_n\cap G_e).\\
  \end{aligned}
\end{equation}
The proof can then be completed as in Proposition \ref{p:good}: we divide by $[G:B_n]$,
let $n\to\infty$ and use $K$-slowness (and a simple coset counting identity) to conclude that
$\lim_n   \dim_K H_j(B_n, K)/[G:B_n] =0$, as claimed.  

The desired decomposition of  $K(B_n,1)=X_n \cup Y_n$ is
obtained as follows. In  paragraph \ref{ss:KG1} we described how to assemble $K(B_n,1)$
from vertex spaces  $K(B_n\cap G_v, 1)$ and edge spaces
$K(B_n\cap G_e, 1)\times [0,1]$ according to the template of the graph 
of groups $B_n\qu \tilde\Gr$, where $\Gr$ is the given graph of groups
with $G=\pi_1\Gr$, and $\tilde\Gr$ is its universal cover. We define $X_n$ to be the
subset consisting of the images of the open edge spaces $K(B_n\cap G_e, 1)\times (0,1)$,
together with  the underlying graph $B_n\backslash \widetilde\Gr$ (the addition of which makes $X_n$ connected).
We define $Y_n'$ to be the union of 
the vertex spaces together with  the underlying graph $B_n\backslash \widetilde\Gr$. We then expand $Y_n'$ into
each edge space  $K(B_n\cap G_e, 1)\times (0,1)$, adding a small cylinder $
K(B_n\cap G_e, 1)\times (0,\epsilon)$ at each end of the edge and obtain $Y_n$ this way. Note that $Y_n'$ deformation retracts
onto $Y_n$, and
$X_n\cap Y_n$ is  easy to describe: it consists of  the underlying graph $B_n\backslash \widetilde\Gr$
 plus two disjoint copies of $K(B_n\cap G_e, 1)\times (0,\epsilon)$ for each edge
where the edge group is $B_n\cap G_e$.

In this construction,   $X_n$ is homotopic to the  1-point union 
of  the graph $B_n\backslash \widetilde\Gr$ and, for each edge $e\in E$, 
disjoint copies of $K(B_n\cap G_e, 1)$ indexed by $B_n\backslash G/G_e$. And
$Y_n$ is  homotopic to the  1-point union  of $B_n\backslash \widetilde\Gr$ and, for each vertex $v\in V$,
disjoint copies of $K(B_n\cap G_v, 1)$ indexed by $B_n \backslash G /G_v$. 
Thus for each $j\ge 2$, omitting the coefficients $K$, we have
$$
 H_j (Y_n) = \Oplus_{v\in V} H_j(B_n\cap G_v)^{[G:B_nG_v]},\ \ \ 
 H_j (X_n) = \Oplus_{e\in E} H_j(B_n\cap G_e)^{[G:B_nG_e]},
 $$
 and
 $
 H_j (X_n\cap Y_n) =  H_j (X_n)\oplus  H_j (X_n)$.
 For $j\ge 3$, the required
 estimate (\ref{sumsH}) is now immediate from the exactness of the Mayer-Vietoris sequence
$$
\dots \to H_j(X_n,K)\oplus H_j (Y_n,K) \to H_j(B_n,K) \to H_{j-1}(X_n\cap Y_n, K)\to\dots
$$

In the calculation of $H_1$, the graph  $B_n\backslash \widetilde\Gr$ contributes
equally to $X_n,\, Y_n$ and $X_n\cap Y_n$, augmenting each of the above
forumlae with a summand $H_1(B_n\backslash \widetilde\Gr, K)$.  
The map on homology
induced by the inclusions of $X_n\cap Y_n$ maps the summand $H_1(B_n\backslash \widetilde\Gr)$ of $H_1(X_n\cap Y_n)$ 
isomorphically to the diagonal of the summand
$H_1(B_n\backslash \widetilde\Gr)\oplus H_1(B_n\backslash \widetilde\Gr)$ in
$H_1(X_n)\oplus H_1(Y_n)$. In particular, the $H_1(B_n\backslash\widetilde\Gr)$ summand
contributes nothing to the kernel of $H_1(X_n\cap Y_n)
\to H_1(X_n)\oplus H_1(Y_n)$, so the exactness of the Mayer-Vietoris
sequence gives us the desired estimate in the case $j=2$ as well.
\end{proof}

 \section{The Proof of Theorem \ref{maintheorem}} \label{s6}

We now turn our attention towards residually free groups that are not limit groups. These form a much
larger and wilder class of groups. In the next section we shall use the structure theory of residually free groups,
as summarised in Section 2, to deduce Theorem \ref{t:rf} from the following more general result, which was stated as
Theorem \ref{maintheorem} in the introduction.

\begin{thm}\label{localT}
Let $G \subseteq G_1 \times \ldots \times G_k$ be a 
subdirect product of residually-finite groups of type $\F$, each of which
contains a normal free subgroup $F_i<G_i$ such that $G_i/F_i$ is 
torsion-free and nilpotent. Assume $F_i\subseteq G\cap G_i$. Let $m < k$ be an integer, let $K$
be a field, and suppose that each $G_i$ is  $K$-slow above dimension
$1$.

If the projection of $G$ to each $m$-tuple of factors 
$G_{j_1} \times \ldots \times G_{j_m} <G$ is of finite index, then
there exists an exhausting normal chain 
$(B_n)$ in $G$ so that for $0 \leq j \leq m$,
$$\limn \frac{ \dim H_j  (B_n, K)}{ [G : B_n]} = 0.$$ 
\end{thm}

The proof of this theorem occupies the whole of this section.  We break it into four steps. First, we
construct exhausting chains $(B_n)$ of finite-index subgroups in  $G$ that are carefully adapted to our purposes. We then state two
technical propositions -- Propositions $\X$ and $\Y$ -- that provide key estimates in the basic spectral sequence argument that we use to prove Theorem \ref{localT}.
This basic spectral sequence argument is carried out in subsection \ref{XYimplyF}, after which we return to Propositions $\X$ and $\Y$ and prove them.
These proofs require extensive spectral sequence calculations that are much more involved than the `basic' one referred to above. 

 We present  an outline of the 
 basic spectral sequence argument immediately  
so that the reader can see where we are going and what we are doing. We assume that the reader is familiar with the Lyndon-Hochschild-Serre (LHS) spectral
sequence in homology associated to a short exact sequence of groups (see \cite{brown}, p.171).
 In the case where the coefficient module is the field $K$ with trivial
action, the $E^2$-page of the LHS spectral sequence associated to $1\to N\to \G\to Q\to 1$ has $(p,q)$-term
 $E_{p,q}^2 = H_p(Q, H_q(N,K))$ and the terms $E_{p,q}^\infty$ with $p+q=s$ are the composition factors of a series for $H_s(\G, K)$.

\subsection{The spectral sequence argument we want to use}
Given $G$ as in Theorem \ref{localT}, we will construct carefully an exhausting normal chain $(B_n)$
so that (after lengthy argument) we obtain enough control on the dimension of the homology groups $H_p(B_n/(B_n\cap N),K)$
and $ H_q(B_n\cap N, K)$ with $p$ and $q$ in a suitable
range, where $N$ is the direct product of the free groups $F_i$ in the statement of Theorem \ref{localT}.
We then argue as follows:

 \begin{lemma} \label{spectral} Let $( B_n )$ be an exhausting normal
 chain for  $G$, let $N\sn G$ be a normal subgroup and let $K$ be a field. Fix an integer $s$.
Suppose for all non-negative integers $\alpha, q$ with $\alpha+q = s$ and every $n$ we have
  $\dim H_{\alpha}(B_n/ B_n \cap N, H_q(B_n \cap N, K)) < \infty$. Suppose further that, for all such
  $(\alpha,q)$,
  $$
 \limn \frac  {\dim H_{\alpha}(B_n/ B_n \cap N, H_q(B_n \cap N, K))}{[G : B_n]}= 0.
 $$
 Then  $$\limn\frac {\dim H_s  (B_n, K)}{[G : B_n]}  = 0.$$
 \end{lemma}
 
 \begin{proof} The LHS spectral sequence
 $$
 E_{\alpha,q}^2 = H_{\alpha}(B_n/ B_n \cap N, H_q(B_n \cap N, K))
 $$
 converges via $\alpha$ to $H_{\alpha+q}(B_n, K)$. Hence $$\dim H_s(B_n, K) \leq \sum_{\alpha+q = s} \dim H_{\alpha}(B_n/ B_n \cap N, H_q(B_n \cap N, K)).$$
 \end{proof}

\subsection{Special filtrations of subdirect products}

We remind the reader that $G<G_1\times\dots\times G_k$ is called a {\em subdirect product} of the groups $G_i$ if it projects onto each factor, i.e.
$\p_i(G) = G_i$ for all $1 \leq i \leq k$. It is called a {\em full} subdirect product if, in addition, $G\cap G_i\neq 1$ for all $1 \leq i \leq k$.

\smallskip

{\em Notation:} Given a direct product $G_1 \times \ldots \times G_k$ and indices $\I=\{i_1,\dots, i_m\}$ with 
$1\leq i_1 < \ldots < i_m \leq k$ we denote the canonical projection by
$$
\p_\I  : G_1 \times \ldots \times G_k \to
G_{i_1} \times \ldots \times G_{i_m},$$
or $\p_{i_1, \ldots , i_m}$, if it is appropriate to be more expansive. 

For a group $H$ and an integer $d$, we write $H^{[d]}$ to denote the subgroup generated by $\{h^d \mid h\in H\}$.

 \begin{lemma} \label{l:filter} Let  $\Gamma=G_1\times\dots\times G_k$ be a direct product of  finitely-generated residually-finite groups.
 Let $G<\Gamma$ be a subdirect product.
 Assume that for all $1 \leq j \leq k$  there is a free group $F_j<G\cap G_j$ that is normal in $G_j$
 with $G_j / F_j$ torsion-free and nilpotent. 
Let $N = F_1 \times \ldots \times F_k \subseteq G$.
Let $m\le k$ be an integer and assume that
 $\p_{j_1, \ldots, j_m}(G)$ has finite index in $G_{j_1} \times \ldots \times G_{j_m}$ for all $ 1 \leq j_1 < \ldots < j_m \leq k$.

Then, one can exhaust $G$ by a chain of finite-index normal subgroups  $( B_n )$ such that :
 
\begin{enumerate}
\item $B_n \cap N = (B_n \cap F_1) \times \ldots \times (B_n \cap F_k)$ for every $n$; 

\item $\bigcap_n B_n N = N$; 
 
\item $\bigcap_n \p_j(B_n) = 1$ for all $1 \leq j \leq k$;
 
\item $B_n \cap F_j = \p_j(B_n) \cap F_j$ for all $n$ and all $1 \leq j \leq k$; 
 
\item $F_j=\bigcap_n \p_j(B_n) F_j $ for all $1 \leq j \leq k$;
 
\item there exists a positive integer $\delta$ such that for all $n$ and $1 \leq j_1 < \ldots < j_m \leq k$
$$ 
\left(\prod_{i=1}^m \p_{j_i}(B_n)^{[\delta]}\right) \left(B_n \cap \prod_{i=1}^m F_{j_i} \right)
\subseteq
\p_{j_1, \ldots, j_m}(B_n) \subseteq \prod_{i=1}^m \p_{j_i}(B_n)
$$
and
$
\left(\prod_{i=1}^m \p_{j_i}(B_n)^{[\delta]}\right) \left(B_n \cap \prod_{i=1}^m F_{j_i} \right)$ is a normal subgroup of finite index in 
$\prod_{i=1}^m \p_{j_i}(B_n)$ such that  the quotient group is a subquotient of $\Gamma / N$.
 \end{enumerate} 

Furthermore, if every $G_j$ is residually $p$-finite for a fixed prime $p$  then every $B_n$ can be chosen to have $p$-power index in $G$.
 \end{lemma}
 
 \begin{proof} Every torsion-free
 finitely generated nilpotent group is residually $p$ (see \cite{Baumslag}). Thus we may exhaust each $G_j/F_j$ by a sequence of normal subgroups
 of $p$-power index, and pulling these back to $G_j$ we obtain normal subgroups of $p$-power index 
  $( S_{j,i} )_i$ in $G_j$  such that $\cap_i S_{j,i} F_j = F_j$. Next,
since $G_j$  is a residually finite  there is a filtration $( K_{j,i} )_i$ of $G_j$ by normal subgroups of finite index so that  $\bigcap_i K_{j,i} = 1$ and $K_{j,i} \subseteq S_{j,i}$. If $G_j$ is residually finite $p$-group we can assume that $K_{j,i}$ has $p$-power index in $G_j$. In any case,
$
F_j \subseteq \bigcap_i K_{j,i} F_j \subseteq \bigcap_i S_{j,i} F_j = F_j,
$
hence
\begin{equation} \label{star11}
 \bigcap_i K_{j,i} F_j = F_j.
\end{equation}
Define $$A_i = K_{1,i} \times \ldots \times K_{k,i} \hbox{ and } \widetilde{B}_i = A_i\cap G.
$$
 Then by (\ref{star11})
\begin{equation} \label{equ1}
\bigcap_i A_i N = \bigcap_i (K_{1,i} F_1 \times \ldots \times
K_{k,i} F_k) = (\bigcap_i K_{1,i} F_1) \times \ldots \times
(\bigcap_i K_{k,i} F_k) =$$ $$  F_1 \times \ldots \times F_k = N
\end{equation}
and since $N \subseteq G$ and \begin{equation} \label{inclusion101} K_{j,i} \cap F_j \subseteq A_i \cap F_j = A_i \cap G \cap F_j = \widetilde{B}_i \cap F_j \end{equation} we have
 $$\widetilde{B}_i \cap N = A_i \cap G \cap N = A_i \cap N =
 (K_{1,i} \cap F_1)\times \ldots \times (K_{k,i} \cap F_k)
 \subseteq $$ 
\begin{equation} \label{star12}
(\widetilde{B}_i \cap F_1) \times \ldots \times (\widetilde{B}_i \cap F_k) \subseteq \widetilde{B}_i \cap N.
 \end{equation}

 Writing $n_i$ for the exponent of the finite group $G / \widetilde{B}_i$ we have 
 \begin{equation} \label{equ2} G^{[n_i]} \subseteq
 \widetilde{B}_i.\end{equation}
 (Note that if $K_{j,i}$
has $p$-power index in  $G_j$ for all $j$ then $n_i$ is a power of $p$.)

Inductively, for each $i$ we choose
 $s_i \in n_i \mathbb{Z} $ so that $s_i$ divides $s_{i+1}$ and $s_i$ is divisible by $p^{a_i}$ where $a_i$ goes to infinity as $i$ goes to infinity. (If
 all $n_i$ are powers of $p$ then we choose $s_i$ to be powers of $p$.)
 We will impose some extra conditions on $s_i$ later.
 
 Finally we are able to define \begin{equation} \label{equ3} B_i := G^{[s_i]} (\widetilde{B}_i \cap N).\end{equation}
From (\ref{equ2}) and (\ref{equ3}) we have
 \begin{equation}  \label{eqNN1}  B_i \subseteq \widetilde{B}_i \hbox{ and } B_i \cap N = \widetilde{B}_i \cap N.
 \end{equation}
 And since $(G/N)^{[s_i]}$ has finite index in the nilpotent $G/ N$ while
 $\widetilde{B}_i \cap N$ has finite index in $N$, we see that $B_i<G$ is a {\em normal subgroup of finite index}.
 \smallskip
 
{\bf{ (1).}} From (\ref{star12}) and (\ref{eqNN1}) we have
$$ B_i \cap N = \widetilde{B}_i \cap N = (\widetilde{B}_i \cap F_1) \times \ldots \times (\widetilde{B}_i \cap F_k),$$
and
 \begin{equation} \label{inclusion102}
  B_i \cap F_j = (B_i \cap N) \cap F_j = (\widetilde B_i \cap N) \cap F_j 
  = \widetilde{B}_i \cap F_j.\end{equation}
  Then
  \begin{equation} \label{auxiliar1} B_i \cap N  = \widetilde{B}_i \cap N = ({B}_i \cap F_1) \times \ldots \times ({B}_i \cap F_k).\end{equation}

{\bf{(2).}} We have $B_i \subseteq \widetilde{B}_i \subseteq A_i$, so from (\ref{equ1}) we deduce
$
 N \subseteq \bigcap_i B_i N \subseteq \bigcap_i A_i N = N$
   hence   \begin{equation} \label{auxiliar2} \bigcap_i B_i N = N.\end{equation}
Also,
 \begin{equation} \label{auxiliar3}
 \bigcap_i B_i \subseteq \bigcap_i A_i = (\bigcap_i K_{1,i}) \times \ldots \times (\bigcap_i K_{k,i}) = 1.
 \end{equation}
 Thus $( B_i )$ is an exhausting filtration of $G$.

 {\bf{(3).}}  Consider $\bigcap_i \p_j(B_i)$. Note that  by (\ref{star12}) and (\ref{equ3}) $$\p_j(B_i) = \p_j(G^{{[s_i]}}) \p_j(\widetilde{B}_i \cap N) =
 G_j^{{[s_i]}} \p_j((K_{1,i} \cap F_1) \times \ldots
 \times (K_{k,i} \cap F_k)) =$$ 
\begin{equation} \label{auxiliar4.1}
 G_j^{{[s_i]}} (K_{j,i} \cap F_j) \subseteq G_j^{{[s_i]}} N.
\end{equation}
 Recall that $G/ N$ is a finitely generated torsion-free nilpotent group, so is residually $p$-finite and
 $\bigcap_t (G/N)^{[t]} =
 1$, whenever $t$ runs through an increasing sequence of $p$-powers. Then
 by (\ref{auxiliar4.1})
 $
 \bigcap_i \p_j(B_i) \subseteq \bigcap_i G_j^{[s_i]} N \subseteq \bigcap_i G_j^{[p^{a_i}]} N \subseteq N,$ so
 \begin{equation} \label{eqNN2}  \bigcap_i \p_j(B_i) = \bigcap_i (\p_j(B_i) \cap N).
 \end{equation}
 Note that by (\ref{auxiliar4.1})
 \begin{equation} \label{eqNN3}
 \p_j(B_i) \cap N = (G_j^{[s_i]} (K_{j,i} \cap F_j)) \cap N =$$ $$
 (G_j^{[s_i]} \cap N) (K_{j,i} \cap F_j) = (G_j^{[s_i]} \cap F_j) (K_{j,i} \cap F_j).
 \end{equation}
 Now we specify the choice of $s_i$ more tightly, multiplying our original choice by the exponents of the finite group $G_j/K_{j,i}$
 if necessary to ensure that for all $i$ and all $1\le j\le k$ we have  
\begin{equation}
G_j^{[s_i]}  \subseteq K_{j,i}.
\end{equation}
 Then by (\ref{eqNN3})
 \begin{equation} \label{auxiliar5} \p_j(B_i) \cap N =
 (G_j^{{s_i}} \cap F_j) (K_{j,i} \cap F_j) = K_{j,i} \cap
 F_j \subseteq
 K_{j,i},\end{equation}  hence by (\ref{eqNN2}), (\ref{auxiliar5}) and the definition of $K_{j,i}$ $$ \bigcap_i \p_j(B_i) = \bigcap_i (\p_j(B_i) \cap N) \subseteq \bigcap_i K_{j,i} = 1.
  $$
  
{\bf{(4).}} 
 By (\ref{inclusion101}), (\ref{inclusion102}) and (\ref{auxiliar5})
  $$
  K_{j,i} \cap F_j \subseteq \widetilde{B}_i \cap F_j = B_i \cap F_j \subseteq  
\p_j(B_i) \cap N = K_{j,i} \cap F_j
  $$
  and so
  \begin{equation} \label{auxiliar6} 
  B_i \cap F_j = \p_j(B_i) \cap F_j.
  \end{equation}
  
{\bf{(5).}} 
  By (\ref{auxiliar4.1})
  $$
  \bigcap_i \p_j(B_i) F_j = \bigcap_i G_j^{[s_i]} (K_{j,i} \cap F_j) F_j = \bigcap_i G_j^{[s_i]} F_j = F_j.
  $$
  
{\bf{(6).}}  Define
$$Y_i:=   \p_{j_1}(B_i) \times \ldots \times \p_{j_m}(B_i) \text{  and  } X_i=Y_i\cap N.
$$
Note that $Y_i$ is the term on the right in the statement of item (6). We claim
that $X_i$ is equal to the second bracketed term on the left. Indeed, from
 (\ref{auxiliar1}) and (\ref{auxiliar6}) we have
\begin{equation} \label{equ4}
B_i\cap \prod_{t=1}^m F_{j_t}
= \prod_{t=1}^m (B_i \cap F_{j_t}) 
=\prod_{t=1}^m (\p_{j_t}(B_i) \cap F_{j_t})  
= X_i  \vartriangleleft   Y_i.
\end{equation}
To see that the  middle term of item (6) is contained in $Y_i$, we use (\ref{auxiliar1}) and (\ref{equ3}) to calculate:
\begin{equation} \label{eqAB1}
\begin{aligned}
\p_{j_1, \ldots , j_m}(B_i) &= \p_{j_1, \ldots , j_m} (G^{[s_i]}) \p_{j_1, \ldots , j_m} (\widetilde{B}_i \cap N) \\
&= 
 (\p_{j_1, \ldots , j_m} (G))^{[s_i]} \p_{j_1, \ldots , j_m} \big(\prod_{l=1}^k({B}_i \cap F_l)\big)\\ 
& = 
 (\p_{j_1, \ldots , j_m} (G))^{[s_i]} \p_{j_1, \ldots , j_m} \big(\prod_{t=1}^m({B}_i \cap F_{j_t})\big)\\ 
&=(\p_{j_1, \ldots , j_m} (G))^{[s_i]} X_i\\
&\subseteq \big(\prod_{t=1}^m G_{j_t}^{[s_i]}\big)\big(\prod_{t=1}^m (B_i \cap F_{j_t})\big)\\
&= \prod_{t=1}^m G_{j_t}^{[s_i]}(B_i \cap F_{j_t}) = \prod_{t=1}^m \p_{j_t}(B_i) = Y_i. 
\end{aligned}
\end{equation}
At this point we have proved that 
$$  B_i\cap \prod_{t = 1}^m  F_{j_t}  =X_i \subseteq  \p_{j_1, \ldots , j_m}(B_i) \subseteq Y_i$$
and that
$$
\p_{j_1, \ldots , j_m}(B_i) = (\p_{j_1, \ldots , j_m} (G))^{[s_i]} X_i.
$$
So to complete the proof of the inclusions displayed in (6) it only remains
to establish the existence of $\delta$ such that   
 \begin{equation} \label{new111}
\prod_{t = 1}^m \p_{j_t} (B_i)^{[\delta ]} \subseteq (\p_{j_1, \ldots , j_m} (G))^{[s_i]} X_i.
\end{equation} 
By  (\ref{auxiliar4.1}) and since $K_{j_t,i} \cap F_{j_t} \subseteq B_i \cap F_{j_t}$ (see (\ref{inclusion101}))
$$
\p_{j_t}(B_i)^{[\delta ]}(B_i \cap F_{j_t}) = (G_{j_t}^{[s_i]} (K_{j_t,i} \cap F_{j_t}))^{[\delta ]} (B_i \cap F_{j_t}) = (G_{j_t}^{[s_i ]})^{[\delta]} (B_i \cap F_{j_t}).
$$
Hence, since $\prod_{t = 1}^m (B_i \cap F_{j_t}) \subseteq X_i$, 
 (\ref{new111}) is equivalent to
\begin{equation}\label{e:need}
\prod_{t = 1}^m (G_{j_t}^{[s_i ]})^{[\delta]}   \subseteq (\p_{j_1, \ldots , j_m} (G))^{[s_i]} X_i,
\end{equation}
so we will be done if we can find $\delta$ to do this.

To this end, assume for a moment that we have found  $\delta$ such that 
\begin{equation} \label{lect21}
\prod_{t =1}^m (G_{j_t}^{[s_i ]})^{[\delta]}   \subseteq (\p_{j_1, \ldots , j_m} (G))^{[s_i]} N.
\end{equation}
Then, recalling that we chose $s_i$ so that $G_j^{[s_i]} \subseteq K_{j,i}=\pi_{j}(A_i)$, for all $j$, we would have 
$$
\prod_{t =1}^m (G_{j_t}^{[s_i ]})^{[\delta]}   \subseteq ((\p_{j_1, \ldots , j_m} (G))^{[s_i]} N) \cap \prod_{t = 1}^m K_{j_t,i},$$
and since $N$ is normal in $\prod_{j=1}^kG_j$, the right hand side equals
$$
 (\p_{j_1, \ldots , j_m} (G))^{[s_i]} \big(N \cap \prod_{t = 1}^m K_{j_t,i}\big).
 $$
Finally, (\ref{star12}), (\ref{auxiliar1}) assure us that the second term equals $B_i\cap \prod_{t = 1}^m F_{j_t,i}$, which is $X_i$, so we
have completed the required proof of (\ref{e:need}), modulo assumption (\ref{lect21}), which we shall prove using the following simple fact.

\smallskip 

\noindent{\bf Claim:} Let $Q$ be a finitely generated torsion-free  nilpotent group and 
let $Q_0<Q$ be a subgroup of finite index. Then there exists $\delta$ so that for every natural number $s$,
$$(Q^{[s]})^{[\delta]} \subseteq Q_0^{[s]}.
$$

\smallskip
\noindent{\em Proof of Claim.} 
Since every subgroup in a nilpotent group is subnormal, an obvious
induction on $[Q:Q_0]$ reduces us to the case where $Q_0$ is normal with index a prime number $q$.
We must show that the exponent of the group $D_s = Q^{[s]}/ Q_0^{[s]}$ can be bounded independently of $s$. 
As $D_s$ is of bounded nilpotency class this is equivalent to the exponent of the abelianization of $D_s$ having an upper bound independent
of $s$. This abelianization is finitely generated with at most $d(Q)$ generators and each generator has exponent $q$ since $Q^{[q]} \subseteq Q_0$. This completes the proof of the claim. 

\smallskip
Returning to the proof of (\ref{lect21}),
recall that
by hypothesis, for our fixed integer $m\le k$, each projection of the form
 $\p_{j_1, \ldots , j_m} (G) $ has finite index in 
 $G_{j_1}\times \ldots \times G_{j_m}$. Thus we can apply the Claim with $Q = \prod_{t = 1}^m G_{j_t} N/N$ and $Q_0 = \p_{j_1, \ldots , j_m} (G) N/N$. This
 completes the proof that the promised inclusions in the statement of part (6) of the lemma hold.

Note that (\ref{equ4})
establishes the normality and subquotient properties we were required to prove. The finite index property follows from the fact that every finitely generated nilpotent group of finite exponent is finite.
The additional $p$-power property claimed in the last sentence of the statement is easily verified by following the stages of the construction.

 This completes the proof of the lemma.
 \end{proof}

\subsection{Proposition $\X$ and Proposition $\Y$} \label{subsectionFG}

From now until the end of Section \ref{s6} we fix $G$ to be a group satisfying the assumptions of Theorem \ref{maintheorem}.
We will work exclusively with  a special filtration, i.e. a particular exhausting sequence  $( B_i )$ of finite-index normal
subgroups, constructed to  satisfy the conditions of Lemma
  \ref{l:filter}. (The fact that the index
  $[G : B_i]$ can be a power of a fixed prime $p$ if each $G_i$ is residually $p$ will, however, play no role.) We shall maintain the notation of Lemma \ref{l:filter} (in particular $\delta$
  is the integer whose existence was established there).
  We also fix, for the duration of the section, an arbitrary 
  $$\I = \{j_1,\ldots,j_m\}$$
   with $1 \leq j_1 < \ldots < j_m \leq k$, and  define 
$$
\Gamma_\I := G_{j_1}\times\dots\times G_{j_m}
$$
$$
\Lambda_{i} := (F_{j_1} \cap B_i) \times \ldots \times (F_{j_{m}} \cap B_i),
$$
$$S_i :=  \p_{j_1}(B_i)^{[\delta ]} (B_i \cap F_{j_1}) \times \ldots \times \p_{j_m}(B_i)^{[\delta ]} (B_i \cap F_{j_m}),$$ 
$$
D_i := S_i  / \Lambda_{i}
$$
and 
$$C_i := \p_{\I} (B_i) /  \Lambda_{i}.$$
By Lemma \ref{l:filter} (6), $D_i \subseteq C_i$.

\begin{lemma}\label{newL} There is a positive integer $b$, independent of $\I$,  so that for all $i$,
$$
| C_i / D_i | \leq b. 
$$ 
\end{lemma}

\begin{proof} Each $C_i/D_i$ is a finite nilpotent group, so to bound its cardinality it is enough to bound the size of a minimal
generating set, the nilpotency class and the exponent.
Lemma \ref{l:filter}(6) tells us that $C_i / D_i$ is   a subquotient of the finitely generated nilpotent group $(G_1 \times \ldots \times G_k) / N$. 
 Since $(G_1 \times \ldots \times G_k) / N$ has finite rank there is an upper bound, independent of $i$, on the number of elements required to
generate $C_i / D_i$. The nilpotency class of $(G_1 \times \ldots \times G_k) / N$ is an upper bound on the nilpotency class of $C_i / D_i$, and  by Lemma \ref{l:filter}(6) the exponent of
 $C_i / D_i$ divides $\delta$. 
 \end{proof}

The proof of Theorem \ref{maintheorem} depends on the following two technical results about asymptotics of homology groups.
Recall that $m=|\I|$.

\medskip
\noindent{\bf Proposition $\X$.} {\it
 For all positive integers $\alpha$ and $q$ with $\alpha + q \leq m-1$  we have
 $$\limii \frac{1}{[\G_\I : S_i]}\dim H_{\alpha}( D_i, H_q(\Lambda_i,K)) = 0 
 $$
 and for $\alpha + q = m$ we have
 $$\limsup_{i\to\infty} \frac{1}{[\G_\I : S_i]}\dim H_{\alpha}( D_i, H_q(\Lambda_i,K))<\infty.
 $$ }

\medskip
\noindent{\bf  Proposition $\Y$.} {\it Suppose that for $\alpha + q \leq m$ we have
 $$\limsup_{i\to\infty} \frac{1}{[\Gamma_\I : \p_\I(B_i)]}\dim H_{\alpha}( C_i
 , H_q(\Lambda_i,K)) < \infty.
 $$
Then for  $\alpha +q  \leq m$
 $$\limii\frac{1}{[G:B_i]} \dim H_{\alpha}( B_i/(N \cap B_i), H_q(\Lambda_i, K)) = 0 .
 $$}

\smallskip

\begin{remark} The hypothesis of Theorem \ref{maintheorem} that the summands $G_j$ are $K$-slow above dimension $1$ enters the proof of 
Proposition $\X$ in an essential way but is not required in the proof of Proposition $\Y$.
\end{remark} 

\subsection{Proposition $\X$ and Proposition $\Y$ imply Theorem \ref{maintheorem}} \label{XYimplyF}

\noindent
Let $C_i$ and $D_i$ be as in subsection \ref{subsectionFG} and let $b$ be the constant of Lemma \ref{newL}. 
The output of Proposition $\X$ controls the homology of $D_i$ while Proposition $\Y$ requires as input control on the homology of $C_i$.
The following lemma bridges this gap.

\begin{lemma} \label{43210} For every $s\in\mathbb N$, there exists a constant $c=c(s,b)$ so that if
 $M$ is a $K C_i$-module with $\dim H_{\alpha}(D_i, M) < \infty$ for all $\alpha \leq s$, then
$$
\dim H_{\alpha}(C_i, M) \leq c \, \max_{t \leq \alpha} \{ \dim H_{t}(D_i, M) \}.
$$
\end{lemma} 
\begin{proof}
Consider the LHS spectral sequence 
$
H_{\gamma} (C_i / D_i, H_t(D_i, M))$ converging to $H_{\gamma+t}(C_i, M)$.
Since $C_i/ D_i$ is a finite group of order at most $b$ there is a free resolution of the trivial $\mathbb{Z}[C_i / D_i]$-module $\mathbb{Z}$ with finitely generated modules in every dimension. We fix one such resolution for every possible finite group of order at most $b$ and define $c_1$ to be the least upper bound on the number of generators of the free modules up to dimension $s$ in these resolutions.  Then, for all $\gamma \leq s$ we have
$$ \dim H_{\gamma} (C_i / D_i, H_t(D_i, M)) \leq c_1 \dim H_t(D_i, M).
$$
As one passes from the $E_2$ page of the spectral sequence to the $E_\infty$ page, the dimension of the $K$-module in each coordinate does not increase,
so the filtration of $H_{\alpha}(C_i, M)$ corresponding to the antidiagonal $\gamma+t=\alpha$ on the $E_\infty$ page allows us to estimate 
$$
\dim H_{\alpha}(C_i, M) \leq \sum_{\gamma+t = \alpha}  \dim H_{\gamma} (C_i / D_i, H_t(D_i, M)) \leq
c_1  \sum_{\gamma+t = \alpha}  \dim H_t(D_i, M).$$ 
Thus
$$ \dim H_{\alpha}(C_i, M) \leq  (c_1 s)\, \max_{t \leq \alpha} \{ \dim H_{t}(D_i, M) \} \hbox{ for } \alpha \leq s,
$$
and setting $c = c_1 s$ we are done.
\end{proof}

We now turn to the main argument. We are trying to get into a situation where we can apply Lemma \ref{spectral}. From Lemma \ref{newL}
we have 
\begin{equation} \label{equ5}
\frac {[\Gamma_\I : S_i]}{[\G_\I: \p_\I(B_i)]}= [C_i : D_i] \leq b.
\end{equation} 
In the light of this, we can combine 
Proposition $\X$, Proposition $\Y$, using Lemma \ref{43210} with $M = 
 H_q(\Lambda_i, K)$, 
to deduce that for $\alpha
+q  \leq m$,
\begin{equation} \label{eqS1}
\limii \frac{\dim H_{\alpha}(B_i/(N\cap B_i),\, H_q(\Lambda_i, K))}{[G : B_i]} = 0. 
 \end{equation}
Now, by Lemma \ref{l:filter}(1),
 \begin{equation} \label{specialcondition} B_i \cap N = (F_1 \cap B_i) \times \ldots \times (F_k \cap B_i).
 \end{equation} 
Since $F_j \cap B_i$ is a subgroup of a free group $F_j$, it
is free itself and $H_s(F_j \cap B_i, K) = 0 $ for $s \geq
2$. Then by the K\"unneth formula \cite[Thm.~11.31]{Rotman} 
 \begin{equation} \label{qhomology}
 H_{q}(B_i\cap N, K) = 
 \oplus_{\mathcal J}  
 H_1(F_{l_1} \cap B_i, K) \otimes_K \ldots \otimes_K H_1(F_{l_q} \cap B_i, K) 
 \end{equation}
 where the sum is over {\em all} $\mathcal J=\{l_1,\dots,l_q\}$ with $1 \leq l_1 < \ldots < l_{q} \leq k$,
while
\begin{equation} \label{qhomology0}
 H_q(\Lambda_i,K)=  
 \oplus_{\mathcal J\subset\I} 
 H_1(F_{l_1} \cap B_i, K) \otimes_K \ldots \otimes_K H_1(F_{l_q} \cap B_i, K).
 \end{equation} 
 
By (\ref{eqS1}) and (\ref{qhomology0}), for $\mathcal J= \{l_1,
   \ldots,l_q \} \subseteq \I$  
we have
$$
\limii\frac 1{[G : B_i]}
 \dim H_{\alpha}( B_i/N \cap B_i,H_1(F_{l_1} \cap B_i, K) \otimes_K \ldots \otimes_K H_1(F_{l_q} \cap B_i, K) ) = 0. 
$$
As the above holds for all choices of $\I$,  
we deduce from (\ref{qhomology}) that
 $$
\limii\frac 1{[G : B_i]} \dim H_{\alpha}( B_i/N \cap
B_i,H_q(N \cap B_i, K) )   = 0 \hbox{ for all } \alpha + q \leq m. 
 $$
This is the estimate that we need as input for the simple spectral sequence argument Lemma \ref{spectral}, so  the proof of Theorem \ref{maintheorem} is complete.
\qed

\subsection{Proof of Proposition $\Y$} 

In our proof of Proposition $\Y$ we shall need the estimates presented in the following two lemmas. These estimates relate the dimension of homology groups of
nilpotent groups to Hirsch length.
 
 \begin{lemma} \label{boundnilpotenthomology2} For every $q\in\mathbb N$ there is a polynomial $f_q \in \BZ[x]$ 
 so that for every finitely-generated  torsion-free nilpotent group $M$,
 $$
 \dim H_q(M, K) \leq f_q(h(M)),
 $$
 where $h(M)$ is the Hirsch length of $M$.
 \end{lemma}
 
 \begin{proof}
 Let $N$ be a central normal subgroup of $M$. As in our previous arguments,
 from the LHS spectral sequence for the short exact sequence $1\to N \to M \to M / N\to 1$ we have
 $$
 \dim H_q(M, K) \leq \sum_{\alpha + \beta = q}  \dim H_{\alpha} (M/ N, H_{\beta} (N, K)).
 $$
 As $N$ is central in $M$, the action of $M/ N$ (via conjugation) on $N$ (hence $H_* (N, K)$) is trivial, so
\begin{equation}\label{e:h}
 H_{\alpha} (M/ N, H_{\beta} (N, K)) \cong H_{\alpha} (M/ N, K) \otimes_K H_{\beta} (N, K),
\end{equation}
and
 $$
 \dim H_q(M, K) \leq \sum_{\alpha + \beta = q} \dim (H_{\alpha} (M/ N, K) \otimes_K H_{\beta} (N, K)).
 $$
By \cite[Cor. 2.11]{Baumslag}  
  there is a central series $( M_i )_i$ for $M$
with all quotients $M_i / M_{i-1}$ infinite cyclic.
Arguing by induction on the length $s + 1 = h(M)$ of this series and making repeated applications of (\ref{e:h}),
for $M = M_s$  we have
 $$
 \dim H_q (M, K) \leq $$ $$ \sum_{\alpha_0 + \ldots + \alpha_s = q} \dim (H_{\alpha_s} (M_s / M_{s-1}, K)  \otimes_K \ldots \otimes_K H_{\alpha_i} (M_i/ M_{i-1}, K) \otimes_K  \ldots \otimes_K  H_{\alpha_0} (M_0, K))
 $$
But $M_i/M_{i-1}\cong\Z$, so $
 H_{\alpha_i} (M_i/ M_{i-1}, K) =0$ if $\alpha_i>1$ and is $K$ in dimensions $0$ and $1$. 
Thus, defining $M_{-1}=0$ we have
 $$
 \begin{aligned}
  \dim H_q (M, K) \leq & \sum_{\alpha_0 + \ldots + \alpha_s = q} \prod_{0 \leq i \leq s} \dim  H_{\alpha_i} (M_i/  M_{i-1}, K) \\
  \leq & \sum_{\alpha_i \leq 1, \alpha_0 + \ldots + \alpha_s=  q} 1\\
     = & {{s+1}
  \choose q} = f_q(s+1),\\
  \end{aligned}
  $$
  where  $f_q(x) = x (x-1) \ldots (x-q +1)/ q!$.
 \end{proof}

 \begin{lemma} \label{boundsum} Let $Q$ be a  torsion-free finitely generated nilpotent group, $H$  a subgroup of $Q$, $V$  a $K H$-module and $M$ a normal subgroup of $H$ such that $M$ acts trivially on $V$. 
Then, for every integer $s$ there is a constant $\beta$ depending only on $s$ and
 the Hirsch length
$h(Q)$  such that 
 $$
 \dim H_s(H, V) \leq \beta \sum_{0 \leq p \leq s}  \dim H_{p} (H / M, V).
 $$
 \end{lemma}
 
 \begin{proof} 
 Consider the LHS spectral sequence
 $$
 E_{p,q}^2 = H_p(H / M, H_q(M, V))
 $$
 converging to $H_{p+q}(H, V)$. Since $M$ acts trivially on $V$ we have $$H_q(M, V) = H_q(M, K) \otimes_K V,$$ 
 where $H/ M$ acts diagonally on the tensor product. 
 Since $H$ is nilpotent, it acts nilpotently on $M$ (by conjugation) and hence $H$ acts nilpotently on $H_q(M, K)$, so there is a
 filtration of $H_q(M, K)$ by $KH$-submodules such that $H$
 acts trivially on the quotients of this filtration;  we denote these sections $W_1, \ldots , W_j$.
 \begin{equation} \label{equa123}
 \dim H_p(H / M, H_q(M, K) \otimes_K V ) \leq \sum_{1 \leq i \leq j} \dim H_p(H/M, W_i \otimes_K V)
 \end{equation}
 and since $H $ acts trivially on $W_i$ we have that, as a $K [H/M]$-module, $W_i \otimes_K V$ is a direct sum of $\dim (W_i)$ copies of $V$. Thus
 \begin{equation} \label{equa124}
   \dim H_p(H/M, W_i \otimes_K V) \leq  \dim W_i \dim H_p(H/M, V).
 \end{equation}
 Note that $$\sum_i \dim W_i = \dim H_q(M, K),$$ so
 \begin{equation} \label{equa125}
 \sum_i  \dim W_i\, \dim H_p(H/M, V) = \dim  H_q(M, K)\ \dim H_p(H/M, V).
 \end{equation}
 Combining (\ref{equa123}), (\ref{equa124}) and (\ref{equa125}) we have
 \begin{equation} \label{inequality}
 \dim E^2_{p,q} = \dim H_p(H / M, H_q(M, K) \otimes_K V ) \leq \dim  H_q(M, K) \dim H_p(H/M, V).
 \end{equation}
 
 By Lemma \ref{boundnilpotenthomology2}, there is a polynomial $f_q(x)$ depending only on $q$, so that the
 dimension of $H_q(M, K)$ is bounded above by $f_q(h(M))$.
 Let $\beta$ be the maximum of $f_q(z)$, where $0 \leq q
 \leq s$ and $0 \leq z \leq h(Q)$.
 Then by (\ref{inequality})
 $$\dim E_{p,q}^2 \leq \dim  H_q(M, K) \dim H_p(H/M, V) \leq \beta \dim H_p(H/M, V).
 $$
 Finally,
 $$
 \dim H_s (H, V) \leq \sum_{p+q = s}  \dim E_{p,q}^2 \leq \sum_{0 \leq p \leq s} \beta \dim H_p(H/M, V).
 $$ 
 
 \end{proof}
\noindent  
{\bf Proof of Proposition $\Y$.}

\noindent 
We saw in Lemma \ref{l:filter} that for $\I=\{j_1,\dots,j_m\}$ we have
\begin{equation} \label{eqeq1}
\p_{\I} (B_i \cap N) =\Lambda_i:=
(F_{j_1} \cap B_i) \times \ldots \times (F_{j_{m}} \cap B_i).
\end{equation}
$\p_\I$ induces a map
  $$
\rho_{\I} :  B_i / (B_i \cap N) \to C_i = \p_{\I}(B_i) / \p_{\I}(B_i \cap N).
 $$
 This map is surjective and we denote its kernel by $M_i$.
 Define $$W_i =  H_q(\Lambda_i, K).$$
Note that $W_i$ is a $C_i$-module via the conjugation action of $\p_{\I}(B_i)$ and $W_i$ is a $ B_i / (B_i \cap N)$-module via  the map $\rho_\I$.
 
These two incarnations of $W_i$ are what we must understand, since the passage from the input of Proposition $\Y$ to the output involves
  changing $H_{\alpha} (C_i, W_i)$ to $H_{\alpha}
   (B_i / (B_i \cap N), W_i)$, changing
   denominators, and sharpening the limit from finite to zero. 
   
  Concerning the denominators, observe that for $P = \prod_{j \notin \I} F_j$ we have 
   $\p_{\I} (P) = 1$, so $\p_{\I} (P B_i) = \p_{\I} (B_i)$. Then, recalling that 
   $\G_\I:= G_{j_1} \times \ldots \times G_{j_{m}}$, we have
   $$[G : P B_i] \geq  [\p_{\I} (G) : \p_{\I} (B_i)]=\frac{[\G_\I : \p_{\I} (B_i)]}
   { [\G_\I : \p_{\I} (G)]}$$
  and so
  $$ 
\begin{aligned}  
  \frac{[\G_\I : \p_{\I} (B_i)]}{ [G : B_i]} &\leq  \frac{[\G_\I : \p_{\I} (G)] [G : P B_i]}{[G : B_i]}\\
  &= \frac{[\G_\I : \p_{\I} (G)]}{[P : B_i \cap P]}.
  \end{aligned}
     $$ 
    We have assumed that $|\I|= m < k$, so $P \not= 1$ and
     \begin{equation} \label{limitlimitnew}
     0 \leq \limii\frac{[\G_\I : \p_{\I} (B_i)]}{[G : B_i]}\leq
      \limii \frac{[\G_{\I} : \p_{\I} (G)]}{[P : B_i \cap P]} = 0.
     \end{equation}
     
 By definition, $M_i$ is the kernel of $\rho_\I$ and hence acts trivially on $W_i$.     
By hypothesis, $G/N$ is torison-free nilpotent and $M_i$ is a subgroup of $B_i / (B_i \cap N) \cong B_i N / N \subseteq G / N$.  
We apply Lemma \ref{boundsum} with $H = B_i / (B_i \cap N)$, $Q = G / N$ and $M = M_i$
to find a constant $\beta$ depending on $h(Q)$ and $m$ so that for  $C_i = H / M_i$ and all $\alpha \leq m$
$$
\dim H_{\alpha}(B_i/(B_i \cap N), W_i) \leq  \beta \sum_{0 \leq p \leq \alpha} \dim H_{p}(C_i, W_i).
$$
Then
\begin{equation} \label{equat111}
\begin{aligned}
\frac{\dim H_{\alpha}(B_i/(B_i \cap N), W_i)}{[G : B_i]} & \leq
 \frac{\beta}{[G
: B_i]}{\sum_{0 \leq j \leq \alpha} \dim H_{j}(C_i, W_i)} \\
& = \beta \ \frac{\sum_{0 \leq j \leq
\alpha} \dim H_{j}(C_i, W_i)}{ [\G_\I : \p_{\I} (B_i)]}\ \frac {[\G_\I : \p_{\I} (B_i)]}{[G: B_i]}.
\end{aligned}
\end{equation}
Since for  $\alpha + q \leq m$ we have
 $$\limsup_{i\to\infty} \frac{\dim H_{\alpha}( C_i, W_i)}{[\G_\I : \p_{\I} (B_i)]} < \infty 
 $$
we obtain by (\ref{limitlimitnew}) and (\ref{equat111}) that
 $$\limii \frac{\dim H_{\alpha}( B_i/N \cap B_i, W_i)}{[G : B_i]} = 0 \hbox{ for } \alpha +q  \leq m
 $$
as required.
\qed

 \subsection{Proof of Proposition $\X$}

We need the following weak form of slowness for nilpotent groups. This is a special case of  \cite[Thm.~0.2(ii)]{LLS}
and can be proved by a straightforward induction on Hirsch length (cf.~Lemma \ref{l:zslow}). 

\begin{lemma} \label{betti0} If $N$ is a finitely generated nilpotent group and $K$ is a field
then $N$ is $K$-slow,  i.e.
for every exhausting normal chain $(B_i )$ and every $m
\geq 0$
$$\limii \frac{\dim H_{m}(B_i, K)}{[N : B_i]} = 0.$$
\end{lemma}
 
The direct summands $G_j$ in the statement of Theorem \ref{maintheorem} are assumed to satisfy
the hypotheses of the following lemma.  

 \begin{lemma} \label{lemma(s-1,1)} Let $K$ be a field, let $G$ be a finitely generated
 residually finite group that is $K$-slow above dimension $1$
 and let $F$ be a normal subgroup of $G$ such that $G/ F$ is nilpotent and $F$ is free. Let $( L_i )_{i \geq 1}$ be an exhausting sequence of normal subgroups of finite index in $G$ such that $\bigcap_i  L_i F  = F$.
 Then, for $q \geq 1$ 
 $$\limii \frac 1{[G : L_i]}\dim H_{q} (L_i / F \cap L_i, (L_i \cap F)^{ab} \otimes_{\BZ} K)  = 0,
 $$
 and
 $$\limsup_{i\to\infty} \frac 1{[G : L_i]}\dim H_{0} (L_i / F \cap L_i, (L_i \cap F)^{ab} \otimes_{\BZ} K)  <\infty.
 $$
 
 \end{lemma}
 
 \begin{proof}
 $K$-slowness means that for $q\ge 2$
 \begin{equation}\label{eslow}  \limii \frac{\dim H_q(L_i, K)}{[G : L_i]} = 0.
\end{equation}
And since $\dim H_1(L_i, K)$ is bounded above by $d(L_i)$, the number of generators required to generate $L_i$,
\begin{equation}\label{1case}
\limsup_{i\to\infty} \frac{\dim H_1(L_i, K)}{[G : L_i]} <\infty.
\end{equation}

The group that we must understand is
$$H_{q} (L_i / (F \cap L_i),\, (L_i \cap F)^{ab} \otimes_{\BZ} K) = H_{q} (L_i /(F \cap L_i),\, H_1(L_i \cap F, K))$$
which is the $E^2_{q,1}$ term of the LHS  spectral sequence
 $$
 E^2_{\alpha, \beta} = H_{\alpha} (L_i / (L_i \cap F), \, H_{\beta}(L_i \cap F, K))
 $$
 converging via $\alpha$ to $H_{\alpha + \beta}(L_i, K)$. We denote the differentials
 $$d^k_{\alpha, \beta} : E_{\alpha, \beta}^k \to  E_{\alpha -k, \beta +k-1}^k.
 $$
 Since $F$ is free $E^k_{\alpha, \beta} = 0$ for $ \beta \notin \{ 0, 1 \}$, so the sequence stabilizes on the $E^3$ page 
 and 
$$E^\infty_{q,1} = E^3_{q,1} = \coker d^2_{q+2,0}$$
is a direct summand of $H_{q+1}(L_i,K)$. Thus
\begin{equation}\label{ekey}
\begin{aligned}
\dim E^2_{q,1}  &\le \dim E^3_{q,1} + \dim E^2_{q+2,0}\\
& \le \dim  H_{q+1}(L_i,K) + \dim H_{q+2}(\bar{L}_i,K),
\end{aligned}
\end{equation}
where  
 $\bar{L}_i:=L_iF/F  \cong L_i / (L_i \cap F)$. By hypothesis, $(\bar{L}_i)$   is an exhausting sequence of normal subgroups of 
 finite index in $G/F$, so by Lemma \ref{betti0} for $s \geq 0$ we have
\begin{equation}\label{equo}
\limii \frac{\dim H_s(\bar{L}_i, K)}{[G/F :\bar{L}_i]} = 0.
\end{equation}
 Now, $[G/F :\bar{L}_i] \le [G:L_i]$, so dividing through (\ref{ekey}) by $[G:L_i]$ and letting $i\to\infty$, from (\ref{eslow}) and (\ref{equo})
 for $q\ge 1$ we get
 $$
 \limii \frac{\dim E^2_{q,1}}{[G:L_i]} =0,
 $$
 while for $q=0$ using (\ref{1case}) we get
 $$
 \limsup_{i\to\infty}  \frac{ \dim E^2_{0,1}}{[G:L_i]} \le 
\limsup_{i\to\infty} \frac{\dim H_1(L_i, K)}{[G : L_i]} <\infty,
$$
as required.
\end{proof}

  \medskip
\noindent{\bf Proof of Proposition $\X$}

\medskip To simplify the notation, we relabel so that $\I=\{1,\dots,m\}$, and for each $j\in\I$ we define
$$A_{j,i} := \p_j(B_i)^{[\delta ]} (B_i \cap F_j).$$
Recall from (\ref{eqAB1}) that $\p_j(B_i) = G_j^{[s_i]} (B_i \cap F_j)$ and hence
$$
A_{j,i} =  (G_j^{[s_i]} (B_i \cap F_j))^{[\delta ]} (B_i \cap F_j) = (G_j^{[s_i]})^{[ \delta]} (B_i \cap F_j)
$$
is a normal subgroup of finite index in $G_j$. From Lemma \ref{l:filter} (5) and (6) we have
\begin{equation} \label{intersection4}  \bigcap_{i \geq 1} A_{j,i} F_j = F_j \hbox{ for all }1
\leq j \leq k \end{equation} 
and
\begin{equation} \label{equ10}
  A_{j,i} \cap F_j  =  B_i \cap F_j.
  \end{equation}
And the notation that we used to state Proposition $\X$ was (dropping the
subscript $\I$)
$$
\G=G_1\times\dots\times G_m, \ \  \Lambda_i = (A_{1,i} \cap F_1) \times \ldots \times (A_{m,i} \cap F_m),
$$
$$
S_i = A_{1,i} \times \ldots \times A_{m,i},\ \text{ and } D_i = S_i/\Lambda_i.
$$
We shall prove Proposition $\X$ be examining the LHS spectral sequences for the short exact sequences
$1\to \Lambda_i \to S_i\to D_i\to 1$. The $E^2_{\alpha,\beta}$ term of this spectral sequence is
$$
\cE_i(\alpha,\beta) :=   H_{\alpha} (D_i, H_{\beta} (\Lambda_i, K))
  $$
 and what we must prove is that
$$
 \limii \frac{\cE_i(\alpha,\beta)}{[\Gamma:S_i]} = 0 \hbox{ for } \alpha +
\beta \leq m-1
$$
and for $\alpha+\beta=m$
$$ \limsup_{i\to\infty} \frac{\cE_i(\alpha,\beta)}{[\Gamma:S_i]} <\infty.
$$
We shall always assume that $\beta \leq m$.
  The $F_j$, being free, have homological dimension $1$, so by K\"unneth formula 
  $$H_{\beta} (\Lambda_i, K) \cong  \oplus_{1 \leq j_1 < j_2 < \ldots < j_{\beta} \leq m }
  (A_{j_1,i} \cap F_{j_1})^{ab} \otimes \ldots \otimes
  (A_{j_{\beta},i} \cap F_{j_{\beta}} )^{ab} \otimes K.
  $$
  (Here, and throughout, tensor products are over $\mathbb{Z}$ unless indicated otherwise.)
 Thus
  $$\cE_i(\alpha, \beta) \cong \oplus_{1 \leq j_1 < j_2
    < \ldots < j_{\beta} \leq m } W_{\alpha, j_1, \ldots ,
    j_{\beta},i}, $$
where
  $$
  W_{\alpha, j_1, \ldots , j_{\beta},i} = H_{\alpha} (D_i, (A_{j_1,i} \cap F_{j_1})^{ab} \otimes \ldots \otimes (A_{j_{\beta},i} \cap F_{j_{\beta}})^{ab} \otimes K).
  $$
  We will be done if we can show, for fixed $\alpha\ge 0$, that
  $$
 \limii \frac{\dim W_{\alpha, j_1, \ldots , j_{\beta},i}}{[\Gamma:S_i]} = 0   \hbox{ for } \beta < m
 $$
 and 
 $$\limsup_{i\to\infty}   \frac{\dim W_{\alpha, j_1, \ldots , j_{\beta},i}}{[\Gamma:S_i]}  < \infty \hbox{ for } \beta = m.
 $$
 To prove this
  let us fix one sequence $1 \leq j_1 < j_2 < \ldots < j_{\beta} \leq m $.
  Without loss of generality we can assume that $j_i = i$ for all $ 1 \leq i \leq \beta$.
  Define
  $$R_{1,i} = A_{1,i} / (A_{1,i} \cap F_1) \times \ldots \times A_{\beta,i} / (A_{\beta,i} \cap F_{\beta})$$  and $$ R_{2,i} = 
   A_{\beta+1,i} / (A_{\beta+1,i} \cap F_{\beta+1}) \times \ldots \times A_{m,i} / (A_{m,i} \cap F_m).
   $$
   Set $$V_i = (A_{1,i} \cap F_1)^{ab} \otimes \ldots \otimes (A_{{\beta},i} \cap F_{{\beta}})^{ab} \otimes K.$$
We will need the following generalised version of the  K\"unneth formula: for $k=1,2$, let
$T_k$ be a group and let $M_k$ be a $K[T_k]$-module, then
$$
H_{\alpha} (T_1 \times T_2, M_1 \otimes_K M_2) = 
\oplus_{\alpha_1 + \alpha_2 = \alpha} H_{\alpha_1} (T_1, M_1) \otimes_K H_{\alpha_2}(T_2, M_2).
$$
To prove this formula, one takes a deleted projective resolution ${\mathcal P}_k$  of $M_k$ as a $K[T_k]$-module 
and observes that the complex ${\mathcal P}_1 \otimes_K {\mathcal P}_2$ is a deleted projective resolution of $M_1 \otimes_K M_2$ as a $K[\ T_1 \times T_2]$-module (see  \cite[Thm.~10.81]{Rotman} for details).

With this formula in hand, we have
   \begin{equation} \label{dimkunneth}
   W_{\alpha, 1, \ldots , {\beta},i} = H_{\alpha} (R_{1,i} \times R_{2,i}, V_i) =
   \oplus_{\alpha_1 + \alpha_2 = \alpha} H_{\alpha_1} (R_{1,i}, V_i) \otimes_K H_{\alpha_2}(R_{2,i}, K)
   \end{equation}
   and if we define $\Sigma_{k_1, \ldots , k_{\beta},i}$ to be
$$H_{k_1} ( A_{1,i} / (A_{1,i} \cap F_1), (A_{1,i} \cap
    F_1)^{ab} \otimes K ) \otimes_K \ldots \otimes_K
    H_{k_{\beta}} ( A_{\beta,i} / (A_{\beta,i} \cap
    F_{\beta}), (A_{\beta,i} \cap F_{\beta})^{ab} \otimes K ) $$
  then
 \begin{equation} \label{oxford3}
    H_{\alpha_1} (R_{1,i}, V_i)  \cong \oplus_{k_1 + \ldots + k_{\beta} = \alpha_1} \Sigma_{k_1, \ldots , k_{\beta},i}.
 \end{equation}
  Thus
$$  \dim  H_{\alpha_1} (R_{1,i}, V_i)  = $$
 \begin{equation} \label{dimkunneth2}
\sum_{k_1 + \ldots +
    k_{\beta} = \alpha_1}  \prod_{1 \leq t \leq \beta} \dim
  H_{k_t} ( A_{t,i} / (A_{t,i} \cap F_t), (A_{t,i} \cap
  F_t)^{ab} \otimes K ). 
  \end{equation}
  Lemma \ref{lemma(s-1,1)} assures us that
\begin{equation}\label{oxford4}
   \limsup_{i\to\infty} \frac{\prod_{1 \leq t \leq \beta}
   \dim H_{k_t} ( A_{t,i} / (A_{t,i} \cap F_t), (A_{t,i}
   \cap F_t)^{ab} \otimes K ))}{[G_1 \times \ldots \times G_{\beta} : A_{1,i} \times \ldots \times A_{\beta,i} ]} <\infty.
   \end{equation}
   And by combining (\ref{dimkunneth2}) and (\ref{oxford4}) we deduce that
   \begin{equation} \label{oxford5}
   \limsup_{i\to\infty} \frac{\dim  H_{\alpha_1} (R_{1,i}, V_i)}{[G_1 \times \ldots \times G_{\beta} : A_{1,i} \times \ldots \times A_{\beta,i} ]} < \infty.
   \end{equation}
   If $\beta = m$ then $R_{2,i} = 1$, so (\ref{oxford5}) completes the proof in this case.
   
   \medskip
   
   Henceforth we assume that $\beta < m$. In this case, $R_{2,i}$ is non-trivial. Indeed, by 
  (\ref{equ10})
$$R_{2,i}\cong (A_{\beta+1,i} F_{\beta+1} /  F_{\beta+1}) 
\times \ldots \times (A_{m,i}F_{m} /  F_{m})$$ is a normal subgroup of finite index in the infinite nilpotent group $\widetilde{N}=(G_{\beta+1} / F_{\beta+1}) \times \ldots \times 
   (G_m / F_m)$. Moreover (\ref{intersection4}) tells us that these subgroups of finite index exhaust $\widetilde{N}$, so we can apply  Lemma \ref{betti0}  to deduce that
   $$
    \limii \dim  H_{\alpha_2} (R_{2,i},
    K) /  [G_{\beta +1} \times \ldots \times G_{m} :
    (A_{\beta+1,i} F_{\beta+1}) \times \ldots \times (A_{m,i}F_{m}) ] = 0 $$
    for every $\alpha_2 \geq 0$, hence
\begin{equation} \label{equ175} \limii \dim  H_{\alpha_2} (R_{2,i},
    K) /  [G_{\beta +1} \times \ldots \times G_{m} :
    A_{\beta+1,i}  \times \ldots \times
    A_{m,i} ] = 0.\end{equation}
    
We now have all that we need to complete the proof. From    (\ref{dimkunneth}) we have
$$
\dim W_{\alpha,1,\dots ,\beta,i} = \sum_{\alpha_1+\alpha_2=\alpha} \dim H_{\alpha_1}(R_{1,i}, V_i). \dim  H_{\alpha_2}(R_{2,i}, K).
$$
We divide both sides by 
$$
[G_1\times\dots\times G_m : A_{1,i} \times \ldots \times A_{m,i}] = \prod_{j=1}^m [G_j : A_{j,i}] = \big(\prod_{j=1}^\beta [G_j : A_{j,i}] \big)\big(\prod_{j=\beta+1}^m [G_j : A_{j,i}] \big)
$$
and let $i\to\infty$. We proved that in (\ref{oxford5}) that, when normalised by $\prod_{j=1}^\beta [G_j : A_{j,i}]$
the terms involving $R_{1,i}$ remain bounded, while (\ref{equ175}) assures us that, when normalised by $\prod_{j=\beta+1}^m [G_j : A_{j,i}]$, the terms
involving $R_{2,i}$ tend to zero. Thus
$$
\limii \frac{\dim  W_{\alpha,1,\dots ,\beta,i}}{[G_1\times\dots\times G_m : A_{1,i} \times \ldots \times A_{m,i}]}=0
$$
as required, and the proofs of Proposition $\X$ and Theorem \ref{maintheorem} are complete.
\qed 

\section{Proof of Theorem \ref{t:rf}}

We recall the statement of Theorem \ref{t:rf}.

\begin{theorem} Let $m\ge 2$ be an integer, let  $G$ be a residually free group of type $\rm{FP}_m$, and 
let $\rho$ be the largest integer such that $G$ contains a direct product of $\rho$
non-abelian free groups.
Then, there exists an exhausting sequence   $(B_n)$
so that for all fields $K$,
\begin{enumerate}
\item
if $G$ is not of type  $\rm{FP}_\infty$, then 
$
\lim_{n}\frac{\dim H_i(B_n, K)}{[G\colon B_n]} = 0$ for all $0\le i\le m;$

\item 
if $G$ is of type $\rm{FP}_\infty$ then for all $j\ge 1$,
$$
\lim_{n\to\infty}\frac{\dim H_j(B_n, K)}{[G\colon B_n]} =  
\begin{cases} (-1)^\rho \chi(G)&\text{if $j=\rho$}\\
0&\text{otherwise}.\\
\end{cases}
$$ 
\end{enumerate}
\end{theorem}  

Theorem \ref{t:bhms} allows us to regard an arbitrary finitely presented residually free group $G$ as a full subdirect product of limit groups
$
G< G_1\times\dots G_k$ and Theorem \ref{t:desi1} tells us that the $G_i$ are free-by-(torsion free nilpotent) as required in Theorem \ref{maintheorem}.
If the $G_i$ are all non-abelian, then item (1) of the above theorem is immediate consequence of  Theorem \ref{maintheorem}
and Theorem \ref{t:desi2}. If one of the factors $G_i$ is abelian, then the intersection of $G$ with this factor is central and free abelian,
so the Euler characteristic of $G$ is zero and the limits in the statement of the theorem are also zero, by virtue of the following
simple lemma. (Note that if $G$ is of type $\rm{F}_m$ then so is the quotient of $G$ by any finitely generated normal abelian subgroup.)

\begin{lemma}\label{l:zslow} Let $1\to Z\to G\overset{p}\to Q\to 1$ be a central extension, with $Z\cong\Z$ and both
$G$ and $Q$ residually finite of type ${\rm{F}}_m$. Then, for all fields $K$ there is an exhausting normal chain $({B}_n)$  in $G$ 
by finite-index normal subgroups so that for all $s\le m$
$$
\limn \frac{\dim H_s(B_n, K)}{[G:B_n]}=0.
$$
\end{lemma}

\begin{proof} The first thing to note is that for any sequence of finite index subgroups $(C_n)$ in $Q$ and any $s\le m$
we have 
\begin{equation}\label{es}
\limsup_{n\to\infty} \  \dim H_s(C_n, K)/[Q:C_n] <\infty.
\end{equation} Indeed $Q$ has a classifying space $BQ$
with finite $m$-skeleton and $\dim H_s(C_n, K)$ is bounded by the number of $s$-cells in the 
$[Q:C_n]$-sheeted covering space of $BQ$ corresponding to $C_n$, which has $r_s[Q:C_n]$ cells of dimension $s$, where $r_s$ is
the number of $s$-cells in $BQ$. 

Let $(A_n)$ and $(D_n)$ be  exhausting chains of finite-index normal subgroups in $G$ and $Q$ respectively.
 Let $B_n=A_n\cap p^{-1}D_n$ and let $\bar{B}_n=p(B_n)$. Then  we have a central extension 
 $1\to Z_n \to B_n\to \bar{B}_n\to 1$ with $Z_n=Z\cap B_n$, 
and from the  LHS spectral sequence we have $\dim H_s(B_n, K) \leq \dim H_s(\bar{B}_n, K) + \dim H_{s-1}(\bar{B}_n, H_1(Z_n,K))$.
But $H_1(Z_n,K)$ is the trivial $K\bar{B}_n$-module $K$, because the action of $B_n$ on $Z$ by conjugation is trivial. Thus
$$
\dim H_s(B_n, K) \leq \dim H_s(\bar{B}_n, K) + \dim H_{s-1}(\bar{B}_n, K).
$$
The proof is completed by dividing this equality through by $[G:B_n]$ and letting $n$ go to infinity, using (\ref{es}) twice and noting
that $[Z:Z_n]=[G:B_n]/[Q, \bar{B}_n]$ tends to infinity.
\end{proof}

It remains to consider the case where $G$ is of type $\rm{FP}_\infty$. Theorem \ref{t:infty} says that $G$ has a subgroup of
finite index $H=H_1\times \ldots \times H_r$ where the $H_i$ are limit groups.  Let $(B_i)$ be an exhausting normal chain
in $G$ such that each $B_i$ is contained in
$H$ and decompose as
$B_i = (B_i \cap H_1) 
\times \ldots \times (B_i \cap H_r)$. Then  by the K\"unneth
formula and  by Corollary B applied for each $H_i$  we have
$$ 
\begin{aligned}
 \limii \frac{\dim
H_j(B_i, K)}{[G : B_i]} &= \frac 1{[G : H]}
 \sum_{j_1 + \ldots + j_r = j} \prod_{1 \leq s \leq r}
 \limii \frac{\dim H_{j_s}(B_i \cap H_s,
 K)}{[H_s : B_i \cap H_s]}\\
 &= \frac 1 {[G : H]}
 \sum_{j_1 + \ldots + j_r = j} \prod_{1 \leq s \leq r}
 (- \delta_{1, j_s} \chi(H_s))\\
 & = \frac 1{[G : H]}(-1)^{r} \delta_{j, r}
 \chi(H)  =  (-1)^{r} \delta_{j, r} \chi(G).
 \end{aligned}
$$
A limit group does not contain a direct product of two or more non-abelian free groups, and every non-abelian limit group contains a non-abelian free
group, so $r=\rho$ unless one or more of the $H_i$ is abelian. If some $H_i$ is abelian, then $\chi(H_i)=\chi(G)=0$.
This completes the proof of Theorem \ref{t:rf}.
\qed

\section{Rank gradient and deficiency gradient for residually free groups}

Let $G$ be a 
group and let $( B_i )$ be an exhausting  normal chain for $G$. We are interested in the rank gradient 
$${\rm{RG}}(G, ( B_i )) = \limii \frac{d(B_i)}{[G : B_i]}$$
and the deficiency gradient $${\rm{DG}} (G, ( B_i )) = \limii \frac{{\rm{def}}(B_i)}{[G : B_i]}.$$
The first limit exists because, for any nested sequence of subgroups of finite index, the sequence 
$(d(B_i)-1)/[G:B_i]$ is non-increasing and bounded below by $0$.
The second limit exists  because 
the sequence $({\rm{def}}(B_i)+1)/[G:B_i]$ is non-increasing
and bounded below by $-{\rm{RG}}(G, ( B_i ))$; cf. proof of Lemma \ref{known2}. Deficiency gradients have recently
been studied by Kar and Nikolov \cite{KN}.

The following lemma is known but we include a simple proof for the reader's convenience.

\begin{lemma} \label{known1} Let $G$ be a finitely generated residually finite group with a finitely generated infinite normal subgroup $N$ such that $G/ N$ is infinite and residually finite.  
\begin{enumerate}
\item There exist exhausting normal chains $( B_i )$ in $G$ such that $( B_i N / N )$ is an
exhausting normal chain in $G/N$.
\item For any such chain,   ${\rm{RG}}(G, ( B_i )) = 0$.
\end{enumerate}
\end{lemma}

\begin{proof} To see that filtrations $( B_i )$ of the desired form exist,
note that if  $( H_i )$ and $(D_i)$ are exhausting normal chains for $G$ and $Q:=G/N$, and if $p : G \to Q$ is the canonical projection, 
then $B_i = H_i\cap p^{-1} (D_i)$ has the required properties. 

The proof of (2) is similar to Lemma \ref{l:zslow}. It relies on the standard fact that if $\Lambda$ is a subgroup of index $k$ 
in a group $\Gamma$, then $d(\Lambda) - 1 \le k(d(\Gamma) - 1)$. 

Let $N_i = N \cap B_i $ and $Q_i = B_i / N_i$.
Then $[G : B_i] = [N : N_i] [Q : Q_i]$ and
 $d(B_i) \leq d(N_i) + d(Q_i) \leq (d(N)-1)[N : N_i] +1 +(d(Q)-1)[Q : Q_i] +1 $. Hence
$$
\begin{aligned}
 \frac{d(B_i)}{[G : B_i]}  
 &\leq \frac 1{[G : B_i]}\big( ((d(N)-1)[N : N_i] +1)\ +  ((d(Q)-1)[Q : Q_i] +1)\big)\\
&\leq \frac{d(N)}{[Q : Q_i]} + \frac{d(Q)}{[N : N_i]} + \frac{2}{[G:B_i]}.
\end{aligned}
$$
Letting $i$ go to infinity we conclude that ${\rm{RG}}(G, ( B_i )) = 0$. 
\end{proof}

Moving up one dimension we have:

\begin{lemma} \label{known2} Let $G$ be a finitely presented residually finite group with a finitely presented  infinite normal subgroup $N$ such that $Q = G/ N$ is infinite and residually finite. Let $( B_i )$ be an exhausting normal chain for $G$, let $N_i = N \cap B_i$ and $Q_i =  B_i N / N$ and assume
that $\bigcap_i Q_i=1$.  

If  
${\rm{RG}}(N, ( N_i)) = 0$  then 
${\rm{DG}}(G, ( B_i )) = 0$.
\end{lemma}

\begin{proof} There is a standard procedure that,
given finite presentations $N_i = \langle X_i \mid R_i \rangle$ and $Q_i = \langle Y_i \mid S_i \rangle$
will construct a finite presentation $B_i=\langle X_i \cup Y_i \mid R_i, \tilde{S}_i, T_i\rangle$ where $|\tilde S_i| = |S_i|$
and $|T_i|=|d(N_i)|.|Y_i|$. 

In more detail, one lifts the canonical projection from the free group $F(Y_i)\to Q_i=B_i/N_i$ to obtain $\mu:F(Y_i)\to B_i$, then
proceeds as follows. For each
$\sigma\in S_i$ one chooses a word $u_\sigma\in F(X_i)$ such that $\mu(\sigma)u_\sigma$ equals $1\in B_i$; then
$\tilde S_i\subset F(X_i\cup Y_i)$ is defined to
consist of the words $\sigma u_\sigma$. To define $T_i$, one first fixes a generating set $\underline X_i$
for $N_i$ with $|\underline X_i|=d(N_i)$. Then, for each $x\in\underline X_i$ one chooses a word $\eta_x\in F(X_i)$ such that $x=\eta_x$
in $N_i$ and for each $y\in Y_i$ one chooses a word
$v_{xy}\in F(X_i)$ so that $v_{xy}=\mu(y)x\mu(y)^{-1}$ in $B_i$. The set $T_i$ consists of the words $y\eta_x y^{-1}v_{xy}^{-1}$.
(This process, although well-defined, is not algorithmic because there is no algorithm that, given a finite presentation,
can  identify a generating set of minimal cardinality for the group presented.) 


Given a finitely presented group $\Gamma=\langle \Upsilon \mid \Sigma \rangle$ and a subgroup $\G_i<\G$ of index $k$, one obtains a 
presentation $\langle \Upsilon_i \mid \Sigma_i\rangle$ with $|\Upsilon_i |-1= k(|\Upsilon |-1)$ 
and $|\Sigma_i |=k|\Sigma|$ by the Reidemeister-Schreier rewriting process.
(Topologically, this amounts to taking a $k$-sheeted covering of the standard 2-complex for $\langle \Upsilon \mid \Sigma\rangle$ and collapsing a 
maximal tree in the $1$-skeleton.) In particular, 
\begin{equation}\label{save}
\limii \frac{|\Upsilon_i|}{[\Gamma : \Gamma_i]}=|\Upsilon | - 1 \text{  and  } \limii \frac{|\Sigma_i |}{[\Gamma : \Gamma_i]} = |\Sigma |.
\end{equation}

We fix finite presentations $N=\langle X\mid R\rangle$ and $Q=\langle Y \mid S\rangle$ and apply the construction of the previous paragraph to
construct presentations $N_i = \langle X_i \mid R_i\rangle$ and $Q_i=\langle Y_i \mid S_i\rangle$, and from these we construct
a presentation $\langle X_i \cup Y_i \mid R_i, \tilde{S}_i, T_i\rangle$ for $B_i$. By definition, the deficiency of this presentation is an upper bound
on the deficiency of $B_i$, so 
$$
\begin{aligned}
 {\rm{def}}(B_i) &\le |T_i| + (|R_i| - |X_i|) + (|S_i|- |Y_i|)\\
 & = d(N_i)\, |Y_i| + (|R_i| - |X_i|) + (|S_i|- |Y_i|).
 \end{aligned}
 $$
Dividing by $[G:B_i] = [N:N_i].[Q:Q_i]$ we get
$$
\frac{1}{[G:B_i]} {\rm{def}}(B_i) \le \frac{d(N_i)}{[N:N_i]} \frac{|Y_i|}{[Q:Q_i]} + \frac{1}{[Q:Q_i]} \frac{|R_i| - |X_i|}{[N:N_i]} +
 \frac{1}{[N:N_i]} \frac{|S_i|- |Y_i|}{[Q:Q_i]}.
$$
Taking the limit $i\to\infty$, the second and third summands on the right tend to zero, by (\ref{save}), while the first tends to   
${\rm{RG}}(N, ( N_i )).(|Y|-1)$, which is zero by hypothesis. Thus ${\rm{DG}}(G, (B_i))  \le 0$.

On the other hand, for any finitely presented group, ${\rm{def}}(\Gamma) \ge  -d(\Gamma)$, because $\G$
has a finite presentation on a set of $d(\G)$ generators. Thus
$-d(B_i) \le {\rm{def}}(B_i)$ and 
$$
-{\rm{RG}}(G, (B_i)) \le  {\rm{DG}}(G, (B_i))  \le 0.
$$
The first term is zero, by Lemma \ref{known1}, so the lemma is proved.
\end{proof}

The following theorem can be viewed as a homotopical version of Theorem \ref{t:rf} in low dimensions. We believe that the condition that $G$ is of type $\FP_3$ is too strong and that type $\FP_2$ is enough (equivalently, finite presentability) as in the homological case, i.e. Theorem \ref{t:rf}, part (1). It 
also seems likely that higher dimensional analogues of this result hold, but we cannot
resolve this problem because we do not have a complete characterisation of the residually free groups (i.e. subdirect products of limit groups) 
that are of type $\FP_m$ for $m \geq 3$. More specifically, the following conjecture remains open:
  for a full subdirect product $H \leq G_1 \times \ldots \times G_n$ with each $G_i$ a non-abelian limit group, $H$ is of type $\FP_m$ for some $m \leq n$  if and only if for every $1 \leq j_1 < \ldots < j_m \leq n$ the index of $\p_{j_1, \ldots, j_m}(H)$ in $G_{j_1} \times \ldots \times G_{j_m}$ is finite.

\begin{theorem}\label{DGFP3} Every  finitely presented residually free group $G$ that is not a limit group admits an exhausting normal chain  $(B_n)$ with respect to which the rank gradient 
$$\RG(G, ( B_n )) = \lim_{n\to\infty} \frac{ d(B_n)}{ [G : B_n]} = 0.$$ 
Furthermore,
if $G$ is of type $\FP_3$ but is not commensurable with a product of two limit groups,
 $(B_n)$  can be chosen so that the deficiency gradient
$\DG(G, ( B_n )) = 0$.  
\end{theorem}

\begin{proof} 
We use Theorem \ref{t:bhms} to embed $G$ as a full subdirect product of limit groups
$G \leq G_1 \times \ldots \times G_m$ such that each of the projections $\p_{j_1,j_2}(G)<G_{j_1}\times G_{j_2}$ is of finite
index; as $G$ is not a limit group, $m\ge 2$.
 There will be an abelian factor $G_i$ if and only if $G$ has a non-trivial (free-abelian) centre.
If $G$ has such a centre, the theorem follows immediately from Lemmas \ref{known1} and \ref{known2}, so henceforth we assume that
the $G_i$ are all non-abelian. 

Proposition 3.2(3) of \cite{BHMS2} states that $M := G \cap (G_1 \times \ldots \times G_{m - 1})$ is finitely generated.
The quotient $G/M=G_m$ is residually finite. Hence, by 
Lemma \ref{known1}, there is an exhausting normal chain $(B_i)$ with ${\rm{RG}}(G, ( B_i )) = 0$, as required.

If $G$ is of type $\FP_3$ but not virtually a product of two limit groups then $m\ge 3$ and
Theorem \ref{t:desi2} tells us for every $1 \leq j_1 < j_2 < j_3 \leq m$ the projection $\p_{j_1, j_2, j_3}(G)$ has finite index in $G_{j_1} \times G_{j_2} \times G_{j_3}$. In particular this holds for $1 \leq j_1 < j_2 < j_3 = m$, so for every $1 \leq j_1 < j_2 \leq m-1$ we see that $p_{j_1, j_2} (M)$ has finite index in $G_{j_1} \times G_{j_2}$. It follows from Theorem \ref{t:bhms} (or the Virtual Surjections Theorem of \cite{BHMS2})
that $M$ is finitely presented. To complete the proof, we want to appeal to Lemma \ref{known2} with $M=N$ and $Q=G_m$, but first we must construct a
chain $(B_i)$ as described in that lemma. To this end, we fix exhausting normal chains $(D_i),\, (A_i)$ and $(\overline Q_i)$
for $G_1,\, G$ and $G_m$, respectively. 

In the first step of the construction we follow the proof of Lemma \ref{known1} with $M\to \p_1(M)$
playing the role of the map $G\to G/N$ that was considered there, and with $H_i:= A_i\cap M$.
As in that proof, ${\rm{RG}}(M, (M_i)) = 0$, where $M_i:=H_i\cap \p_1^{-1}(D_i)$.

Finally, we define $B_i=A_i\cap \p_1^{-1}(D_i) \cap\pi_m^{-1}(\overline Q_i)$. This is
an exhausting normal chain for $G$ with $M_i=B_i\cap M$. Moreover, $Q_i:= B_i/M_i \subset \overline Q_i$, so $\bigcap_iQ_i=1$.
Thus Lemma \ref{known2} applies and the theorem is proved. 
\end{proof}

\begin{remark} The exceptions made in the statement of Theorem \ref{DGFP3} are necessary: if $G_i$, $i = 1,2$, is a limit group then Theorem A(1)
tells us that the rank gradient of $G_i$ is $-\chi(G_i)$, from which it is easy to deduce that if $G=G_1\times G_2$ is a product of two limit groups then its
deficiency gradient is $\chi(G_1)\chi(G_2)$. 
\end{remark}

 \end{document}